\documentclass[onefignum,onetabnum]{siamart190516}
\usepackage[utf8]{inputenc}

\usepackage{amsmath}
\usepackage{amsxtra}
\usepackage{amsfonts}
\usepackage{stmaryrd}
\usepackage[normalem]{ulem}
\usepackage{graphicx}
\usepackage{subfigure}
\usepackage{epstopdf}
\usepackage{url}

\usepackage{thm-restate}

\usepackage{algorithm}
\usepackage{algorithmic}

\usepackage{appendix}
\usepackage{braket}
\usepackage{multirow}

\usepackage{comment}

\DeclareMathOperator*{\argmin}{arg\,min}
\DeclareMathOperator*{\trace}{tr}
\DeclareMathOperator{\rank}{rank}
\newcommand{\dop}{\mathrm{D}}
\newcommand{\reals}{\mathbb{R}}
\newcommand{\irange}[1][k]{[\![#1]\!]}
\newcommand{\rkval}{R}

\newcommand{\cpdp}[1][U]{\llbracket #1^{(1)},\cdots, #1^{(k)}\rrbracket}

\newcommand{\tstar}{\mathcal{T}^{\star}}
\newcommand{\fro}[1]{\|#1\|_{\mathrm{F}}}
\newcommand{\frob}[1][X]{\|#1\|_{\mathrm{F}}}
\newcommand{\frobn}[1]{\left\|#1\right\|_{F}}
\newcommand{\twopartdef}[4]
{\left\{
    \begin{array}{ll}
      #1 & \mbox{if } #2 \\
      #3 & \mbox{ } #4
    \end{array}
  \right.
}
\newcommand{\eg}{e.g.}
\newcommand{\ie}{i.e.}
\newcommand{\rgrad}[1][x]{\mathrm{grad}f(#1)}
\newcommand{\grad}[1][f]{\mathrm{grad}#1}
\newcommand{\trgrad}[1][x]{\mathrm{grad}f(#1)}
\newcommand{\tansp}[1][x]{T_{#1}}
\newcommand{\ttanv}[1][\xi]{#1}
\newcommand{\prodsp}{\reals^{m_1\times R}\times\dots\times\reals^{m_k\times R}} 
 
\newcommand{\man}{\mathcal{M}}

\newcommand{\ttansp}{T_{U}\man}

\newcommand{\retr}[1][x]{\mathcal{R}_{#1}}
\newcommand{\mtotal}{\mathcal{M}}

\newcommand{\po}[1][\Omega]{\mathcal{P}_{#1}}

\newcommand{\rowof}[2]{#1_{#2,:}}
\newcommand{\colof}[2]{#1_{:,#2}}
\newcommand{\rstar}{r^{\star}}

\newcommand{\expe}[1][\Omega]{\mathbb{E}_{#1}}
\newcommand{\hmati}[1][i]{H_{#1,U}}
\newcommand{\ujexi}{(U^{(j)})^{\odot_{j\neq i}}}
\newcommand{\cchol}{C_{\mathrm{chol}}}
\newcommand{\algseq}[1][t]{\{x_{#1}\}_{#1\geq0}}
\newcommand{\sls}{\mathcal{L}}
\newcommand{\truncr}[1][r^{\star}]{\mathbb{T}_{#1}}
\newcommand{\rktc}{\rank_{\mathrm{tc}}}
\newcommand{\rkcp}{\rank_{\mathrm{CP}}}
\newcommand{\ddt}{\frac{\mathrm{d}}{\mathrm{d}t}}
\newcommand{\trp}{T}
\newcommand{\ceil}[1]{\left\lceil #1 \right\rceil}
\newcommand{\mc}[1]{\mathcal{#1}}

\newcommand{\revv}[1]{#1}
\newcommand{\revj}[1]{#1}

\begin{document}

\title{New Riemannian preconditioned algorithms for tensor completion via
polyadic decomposition\thanks{This work was supported by the Fonds de
la Recherche Scientifique -- FNRS and the Fonds Wetenschappelijk Onderzoek --
Vlaanderen under EOS Project no.\ 30468160 \revj{(SeLMA -- Structured low-rank matrix/tensor
approximation: numerical optimization-based algorithms and applications)}. The first author was supported by the FNRS through a FRIA scholarship.}}

\headers{Riemannian preconditioning for tensor completion}{S. Dong, B. Gao, Y. Guan, F. Glineur}

\author{
Shuyu Dong\thanks{ICTEAM, UCLouvain, Louvain-la-Neuve, Belgium
(\{shuyu.dong,~yu.guan\}@uclouvain.be, gaobin@lsec.cc.ac.cn).}
\and Bin Gao\footnotemark[2]
\and Yu Guan\footnotemark[2]
\and Fran\c{c}ois Glineur\thanks{ICTEAM and CORE, UCLouvain, Louvain-la-Neuve, Belgium (francois.glineur@uclouvain.be).} 
}


\maketitle

\begin{abstract}
We propose new Riemannian preconditioned algorithms for low-rank tensor completion
via the polyadic decomposition of a tensor. These algorithms exploit a
non-Euclidean metric on the product space of the factor matrices of the low-rank
tensor in the polyadic decomposition form. This new metric is designed using an
approximation of the diagonal blocks of the Hessian of the tensor completion
cost function, thus has a preconditioning effect on these algorithms. 
We prove that the proposed Riemannian gradient descent algorithm globally converges to a stationary point of the tensor completion problem, with convergence rate estimates using the {\L}ojasiewicz property. 
Numerical results on synthetic and real-world data suggest that the proposed algorithms are more efficient in memory and time compared to state-of-the-art algorithms. 
Moreover, the proposed algorithms display a greater tolerance for overestimated rank parameters in terms of the tensor recovery performance, thus enable a flexible choice of the rank parameter. 
\end{abstract}



\section{Introduction}\label{sec:introduction}
Tensor completion refers to the task of recovering missing values of a
multidimensional array and can be seen as a generalization of the matrix completion problem.
Similar to the approximation of a matrix with low-rank models, the approximation of a tensor can be formulated by a {\it low-rank}\/ tensor model. 
Starting from this idea, 
low-rank tensor completion consists in finding a low-rank approximation of a
tensor based on a given subset of its entries. %
Applications of low-rank tensor completion can be found in many areas, \eg, signal processing for EEG (brain signals) data~\cite{morup2006parallel} and MRI (magnetic
resonance imaging)~\cite{banco2016sampling}, and image and video
inpainting~\cite{bertalmio2000image,liu2012tensor,korah2007spatiotemporal}. %

Depending on different tensor decomposition forms, there are several ways to define the rank of a tensor. 
Low-rank tensor decompositions provide a useful tool for tensor representation and are widely used in tensor completion~\cite{andersson1998improving,chu2009probabilistic,karatzoglou2010multiverse,acar2011scalable,papalexakis2016tensors,tomasi2005parafac,kolda2009tensor}. %
The low-rank tensor decomposition paradigm 
allows for extracting the most meaningful and informative latent structures of a tensor, which usually contain heterogeneous and multi-aspect data. %
The Canonical Polyadic (CP)
decomposition~\cite{hitchcock1927expression,kruskal1977three,kolda2009tensor}, the Tucker 
or the multilinear decomposition~\cite{Tucker1966,DeLathauwer2000msvd,DeLathauwer2000:obr}, and the tensor-train (TT)
decomposition~\cite{oseledets2011tensor,grasedyck2015alternating,phien2016efficient} are among the most fundamental tensor decomposition forms. 
Other variants include hierarchical tensor representations~\cite{da2013hierarchical,rauhut2015tensor,rauhut2017low} and PARAFAC2 models~\cite{rauhut2015tensor}. %

\paragraph{Related work} For the tensor completion problem using the CP
decomposition (CPD), Tomasi and Bro~\cite{tomasi2005parafac} proposed to
used the Levenberg-Marquardt (modified Gauss--Newton) method for third-order tensors, %
in which the rank-$R$ tensor candidate is represented by the
vectorization of all three factor matrices of the CPD; %
Acar et~al.~\cite{acar2011scalable} proposed to use a nonlinear conjugate
gradient algorithm; %
Jain et~al.~\cite{jain2014provable} proposed an alternating minimization 
algorithm, %
which uses a special initialization by 
the Robust Tensor Power Method~\cite{anandkumar2014tensor} allowing for a
guaranteed tensor recovery with a sample complexity lower bound. %
The Tucker decomposition, as a closely related decomposition format, is
also widely used for tensor completion, with more or less different application
purposes; see~\cite{cichocki2015tensor} about the relations and differences
between the Tucker and the CP decompositions. Kressner et~al.~\cite{kressner2014low} exploited the
Riemannian geometry of the rank-constrained search space with the Tucker
decomposition and proposed a Riemannian conjugate gradient (RCG) algorithm for
the tensor completion problem. Kasai and Mishra~\cite{kasai2016low} paid
attention to the data fitting function of tensor completion %
and introduced a preconditioned metric on the quotient space of the
rank-constrained search space with the Tucker decomposition; they used a
Riemannian conjugate gradient algorithm with respect to the proposed metric. 
\revj{Breiding and Vannieuwenhoven~\cite{breiding2018riemannian} proposed a
Riemannian Gauss--Newton method for tensor approximation, which is based on the
{\it Segre manifold} structure of tensors with a bounded rank.} %
\revj{We refer
to~\cite{sorber2013optimization,sorber2015structured,sorensen2019fiber} for more recent work about CP decomposition methods with fully observed or missing entries.}

For a $m_1\times\dots\times m_k$ tensor that is only partially observed, low-rank tensor completion using CP
decomposition %
can be modeled by approximating the given partially observed tensor with a rank-$R$ tensor candidate in the CP decomposition form $\mathcal{T}=\llbracket
U^{(1)},\dots,U^{(k)}\rrbracket$, where $U^{(i)}$ are $m_i\times R$ matrices with full column-rank %
and $\llbracket U^{(1)},\dots,U^{(k)}\rrbracket$ denotes the product of the CP decomposition, which is the sum of the outer products of the respective columns of $U^{(i)}$. %
In such problem formulations, the CP decomposition not only provides a powerful
data representation, but also has an advantageously low
memory requirement---which scales as $O((m_1+\dots+m_k)R)$ for a given CP rank
$R$---compared to other types of models (\eg,~\cite{liu2012tensor}) that involve otherwise a full dense
tensor variable (requiring $O(m_1m_2\dots m_k)$ memory). %
However, the fixed-rank CP decomposition, as well as other tensor decomposition models with a fixed-rank, requires the choice of an appropriate rank value. Since an optimal rank choice is usually unknown in practice, fixing the tensor rank in the CP decomposition-based (as well as Tucker-based) approaches is not an ideal strategy. Unfortunately, the search or estimation of the optimal rank is also hard~\cite{kruskal1989rank}. Therefore, the approach with a fixed CP rank limits the applicability of the completion model, and it is %
natural to study CP decompositions with a rank upper bound (\eg,~\cite{liu2014factor}). %
For example, a rank-increasing approach %
is used (\eg,~\cite{yokota2016smooth}) during the optimization process for the exploration of an optimal rank. %

In this paper, we focus on the polyadic decomposition (PD) approach to tensor completion %
and consider the tensor candidate with a bounded CP rank.
More precisely, the $k^{\text{th}}$-order tensor variable is represented by $k$ factor
matrices in the polyadic decomposition form. Thus, we formulate the tensor completion
problem as follows, %
\begin{equation}\label{prog:main}
\underset{U:=(U^{(1)},\dots,U^{(k)})\in\man}{\text{minimize}}\quad
\frac{1}{2}\fro{\po(\llbracket U^{(1)},\dots,U^{(k)}\rrbracket - \tstar)}^2 + \psi(U),
\end{equation}
where $\llbracket U^{(1)},...,U^{(k)}\rrbracket$ denotes the tensor candidate in
the form of polyadic decomposition, $\tstar$ is the partially observed tensor, $\Omega$ is an index set indicating the observed entries of $\tstar$, $\po$ denotes the orthogonal projector such
that the $(i_1,\dots,i_k)$-th entry of $\po(Z)$ is equal to $Z_{i_1,\dots,i_k}$
if $(i_1,\dots,i_k)\in\Omega$ and zero otherwise, and the search space $\man$ is
a product space defined as 
\begin{equation}\label{def:man}
\man=\reals^{m_1\times R}\times\cdots\times\reals^{m_k\times R} 
\end{equation} 
with a rank parameter $R$, and $\psi:\man\mapsto\reals$ is a regularization function. %
\revj{Note that the search space of~\eqref{prog:main} includes points such
that} 
the actual CP rank of the product $\llbracket U^{(1)},\dots,U^{(k)}\rrbracket$
is smaller than the rank parameter $R$. 
\revj{Hence, the PD representation of the tensor candidate by $U\in\man$ is not necessarily {\it canonical}.} 
\revj{In the context of low-rank tensor completion, the optimal rank of the
tensor candidate is unknown, and thus, it is important to choose 
appropriate parameters of the low-rank model in question. 
One of the approaches (using low-rank tensor decomposition) in previous work is
to explore and adjust the rank parameter $R$
sequentially~\cite{yokota2016smooth}. In the case of the model using the Tucker decomposition with a fixed Tucker rank~\cite{kasai2016low}, the natural way to choosing the Tucker rank is to explore among a range of ($k$-dimensional) trials.}

\revj{In this work, we are interested in choosing a large enough rank
parameter (while maintaining it in the range of values that are much lower
than the tensor dimensions) for the model~\eqref{prog:main}. To alleviate issues where $R$ is overestimated
compared to the rank of the desired solution, we design algorithms that
tolerate rank-deficient factor matrices (or {\it almost} rank-deficient matrices, with close-to-zero tailing singular values). %
The proposed algorithms are shown in numerical results to be able to reach
optimal solutions to~\eqref{prog:main}, given a large enough value of $R$, and
thus, avoid lengthy sequential rank explorations.}

\paragraph{Contributions} 
We design a {\it preconditioned} metric on the search space $\man$ of the polyadic decomposition model~\eqref{prog:main} and propose %
Riemannian gradient descent (RGD) and Riemannian conjugate gradient (RCG) algorithms 
to solve the tensor completion problem.

We prove that the sequence of iterates generated by the RGD algorithm converges to a critical point of the objective function and provide estimates of the convergence rate using the %
{\L}ojasiewicz property.

We test on synthetic data for recovering a partially observed tensor with and without additive noise. The numerical results show that our algorithms outperform the several existing algorithms. %
On the real-world dataset (MovieLens~1M), we %
find that the proposed algorithms are also faster %
than the other algorithms under various rank choices. %
Moreover, the tensor recovery performance of our algorithms is not sensitive to the choice of the rank parameter, in contrast to algorithms based on a fixed-rank model.

\paragraph{Organization} We give the definitions and notation for the tensor
operations and state the main problem in Section \ref{sec:preliminaries}. The proposed algorithms and convergence analysis are presented in Sections~\ref{sec:algorithm}--\ref{sec:convergence}. %
Numerical results are reported in Section~\ref{sec:numerical}. We conclude the paper in Section~\ref{sec:conclusion}.

\section{Preliminaries and problem statement}\label{sec:preliminaries}
In this section, we introduce the definitions and notation involved in the
tensor operations and give a concrete problem formulation of~\eqref{prog:main}. 

The term {\it tensor} refers to a multidimensional array. The dimensionality of a
tensor is described as its order. A $k^{\text{th}}$-order tensor is a $k$-way array, also
known as a $k$-mode tensor. 
We use %
the term {\it mode} to describe operations on a specific dimension (e.g., mode-$k$
matricization).

For a strictly positive integer $n$, we denote the index set $\{1,\dots,n\}$ as $\irange[n]$. The set of $n$-dimensional real-valued vectors is denoted by $\reals^n$. For $\ell\in\irange[n]$, we denote by $\mathbf{e}_{\ell}$ the ($n$-dimensional) vector that indicates $1$ in the $\ell$-th entry: $[\mathbf{e}_{\ell}]_{\ell}=1$ and $[\mathbf{e}_{\ell}]_{j}=0$ for all $j\neq \ell$. 
The set of $k^{\text{th}}$-order real-valued tensors is denoted by $\reals^{m_1\times\dots \times m_k}$. 
The $i$-th row and the $j$-th column of a matrix $A$ are denoted by $\rowof{A}{i}$ and
$\colof{A}{j}$ respectively. 
An entry of a real-valued $k^{\text{th}}$-order tensor $\mathcal{Z}\in {\mathbb{R}^{m_{1}\times\cdots\times m_{k}}}$
is accessed via an $k$-dimensional index 
$[\ell_{j}]_{j=1,\dots,k}$, with 
$\ell_j\in\irange[m_j]$, and is denoted as $\mathcal{Z}_{\ell_{1},\ldots,\ell_{k}}$. 
The inner product of two tensors $\mathcal{Z}^{(1)},\mathcal{Z}^{(2)}\in
{\mathbb{R}^{m_{1}\times\cdots\times m_{k}}}$ %
is defined as follows 
$\left\langle \mathcal{Z}^{(1)},\mathcal{Z}^{(2)} \right\rangle=\sum_{i_{1}=1}^{m_{1}}\cdots \sum_{i_{k}=1}^{m_{k}}\mathcal{Z}^{(1)}_{i_{1},\ldots,i_{k}}\mathcal{Z}^{(2)}_{i_{1},\ldots,i_{k}}$. %
The Frobenius norm of a tensor
$\mathcal{Z}$ is defined~as 
$\frob[\mathcal{Z}]=\sqrt{\langle \mathcal{Z},\mathcal{Z}\rangle}$.

The following definitions are involved in the tensor computations. 
The Kronecker product of vectors $\mathbf{u}=\left[u_{\ell}\right]\in\mathbb{R}^{m_{1}}$ and $\mathbf{v}=\left[v_{\ell}\right]\in\mathbb{R}^{m_{2}}$ results in a vector 
$\mathbf{u}\otimes \mathbf{v} =[u_{1}\mathbf{v}^{\trp}, u_{2}\mathbf{v}^{\trp}, \dots, u_{m_{1}}\mathbf{v}^{\trp}]^{\trp}\in\mathbb{R}^{m_{1}m_{2}}$. %
The Khatri--Rao product of two matrices with the same number of columns 
$U\in\mathbb{R}^{m_{1} \times R}$ and $V\in \mathbb{R}^{m_{2}\times R}$ 
is defined and denoted as $U\odot V = [
    \colof{U}{1} \otimes \colof{V}{1}, \dots, \colof{U}{R} \otimes \colof{V}{R}]\in\reals^{m_{1}m_{2}\times R}$. %
The Hadamard product of two matrices $A$ and $B$ of the same dimensions,
denoted by $A\star B$, is a matrix of the same dimensions by entrywise
multiplications: $[A\star B]_{ij}=A_{ij}B_{ij}$. %
The mode-$\ell$ product of a given tensor $\mathcal{G}\in {\mathbb{R}^{r_{1}\times\dots\times r_{k}}}$ 
with a matrix $U\in\mathbb{R}^{m\times r_{\ell}}$, denoted as $\mathcal{G}\times_{\ell} U$ is a tensor of size 
$r_{1}\times \cdots\times r_{\ell-1}\times m\times r_{\ell+1}\cdots\times
r_{k}$, which has entries %
 \begin{equation*}
  [\mathcal{G}\times_{\ell} U]_{i_{1},\ldots,i_{\ell-1},j,i_{\ell+1,\ldots,i_{k}}}=\sum_{p=1}^{r_{\ell}}U_{j,p}\mathcal{G}_{i_{1},\ldots,i_{\ell-1},p,i_{\ell+1,\ldots,i_{k}}},
 \end{equation*}
 for 
$j\in\irange[m]$, $i_1\in\irange[r_1]$ and so on. %
The mode-$i$ matricization $\mathcal{Z}_{(i)}$ of a tensor $\mathcal{Z}\in
{\mathbb{R}^{m_{1}\times\cdots\times m_{k}}}$, also called the unfolding of $\mathcal{Z}$ along the $i$-th mode, is a matrix of size $m_{i}\times (\prod_{j\neq i}m_{j})$ such that the
tensor element $z_{\ell_{1},\ldots ,\ell_{k}}$ in $\mathcal{Z}$ is identified with the matrix element $[\mathcal{Z}_{(i)}]_{\ell_{i}, r_{i}}$ in $\mathcal{Z}_{(i)}$, where
$r_{i}=1+\sum^{k}_{\substack{n=1, n \neq i}}(\ell_{n}-1)I_{n}$, with $\quad I_{n}=\prod^{n-1}_{\substack{j=1,j \neq i}}m_{j}$. Using the $(\ell_i,r_i)$-indexing, the mode-$i$ matricization of the $k$-dimensional indices in an index set $\Omega\subset\irange[m_1]\times\dots\times\irange[m_k]$ is defined in the same way as in the tensor matricization, and the mode-$i$ matricization of $\Omega$ is an index set $\Omega_{(i)}$ such that
$(\ell_i, r_i)\in\Omega_{(i)} \text{~if and only if~} (\ell_1,\dots,\ell_k)\in\Omega$.

\begin{restatable}[Tucker decomposition]{redef}{deftcdecomp}\label{def:tucker-decomp}
The Tucker decomposition of a
tensor~\cite{Tucker1966,DeLathauwer2000msvd,DeLathauwer2000:obr} is defined with 
a series of mode-$\ell$ products between a core tensor %
$\mathcal{G}\in\mathbb{R}^{r_1\times\dots\times r_k}$ 
and the (orthogonal) factor matrices $U^{(i)}\in\mathbb{R}^{m_{i}\times r_{i}}$ such~that
\begin{equation*}%
\mathcal{Z}=\mathcal{G}\times_{1}U^{(1)}\times_{2}\cdots\times_{k}U^{(k)}.
\end{equation*}
\end{restatable}

A tensor $\mathcal{Z}$ admitting a Tucker decomposition form with
$\mathcal{G}$ and $U^{(1)},\dots,U^{(k)}$ can be unfolded along its $i$-th mode
into the following matricization, 
\begin{equation*}%
\mathcal{Z}_{(i)} = U^{(i)}\mathcal{G}_{(i)}\left(U^{(i-1)}\otimes \cdots
\otimes U^{(1)} \otimes U^{(k)}\otimes \cdots \otimes U^{(i+1)}\right)^{\trp}. 
\end{equation*}
The tensor Tucker rank or multilinear rank~\cite{Tucker1966} %
is defined as 
\begin{equation*}
	\rktc(\mathcal{Z})=(\rank(\mathcal{Z}_{(1)}),\ldots, \rank(\mathcal{Z}_{(k)}) ), 
\end{equation*}
where $\rank(\mathcal{Z}_{(i)})$ denotes the matrix rank, for $i=1,\dots,k$. 

A tensor $\mathcal{Z}\in\reals^{m_1\times\dots\times m_k}$ of the form $\mathcal{Z}=\mathbf{u}^{(1)}\circ \cdots \circ
\mathbf{u}^{(k)}$, %
where $\mathbf{u}^{(i)}\in\reals^{m_i}$ and %
$\circ$ denotes the outer product, is said to be a rank-$1$ tensor, which is also called a simple tensor \cite{hernandez2010simple} or decomposable tensor \cite{hackbusch2012tensor}. %
The CP rank of a tensor $\mathcal{Z}$ is defined as the minimum number of rank-$1$ tensors which sum to $\mathcal{Z}$~\cite{hitchcock1927expression,kruskal1977three}:
\begin{equation*}
 \rkcp(\mathcal{Z})=\min\{R\in\mathbb{Z}_{+}: \exists\{\mathbf{u}_{r}^{(1)},\dots,\mathbf{u}_r^{(k)}\}_{r=1,\dots,R}, \text{~s.t.~} \, \,\mathcal{Z}=\sum_{r=1}^{R} \mathbf{u}_{r}^{(1)}\circ\dots\circ\mathbf{u}_{r}^{(k)}\}.
\end{equation*}

\begin{restatable}[CP decomposition]{redef}{defcpd}\label{def:cp-decomp}
Let $\mathcal{Z}\in\reals^{m_1\times\dots\times m_k}$ be a tensor such that
$\rkcp(\mathcal{Z})=R$. Then there exist $k$ matrices $U^{(i)}\in\reals^{m_i\times R}$ with %
full column-rank, 
for $i=1,\dots, k$, such that 
\begin{equation*}
\mathcal{Z}=\llbracket U^{(1)},\ldots,U^{(k)}\rrbracket=\sum_{r=1}^{R}U_{:,r}^{(1)}\circ \dots \circ U_{:,r}^{(k)},
\end{equation*}
which is referred to as the Canonical Polyadic decomposition~\cite{hitchcock1927expression,kruskal1977three,kolda2009tensor} of $\mathcal{Z}$. %
\end{restatable}
The CP decomposition can be considered as the ``diagonalized'' version of the Tucker decomposition (Definition~\ref{def:tucker-decomp}), in the sense that a tensor that admits a rank-$R$ CP decomposition also admits a Tucker decomposition with
a $R\times\dots\times R$ hypercube as core tensor, whose nonzero
values are only on the diagonal entries. \cref{fig:tucker-cpd} illustrates the Tucker and CP decompositions of a third-order tensor. 
The matricization of a tensor in the CP decomposition form $\mathcal{Z}=\llbracket U^{(1)},\dots,U^{(k)}\rrbracket$ can be
written using the Khatri--Rao product, 
\begin{equation}\label{eq:def-cpd-mat}
\mathcal{Z}_{(i)}=U^{(i)}[(U^{(j)})^{\odot_{j\neq i}}]^{\trp} := U^{(i)}(U^{(k)}\odot \cdots \odot U^{(i+1)} \odot U^{(i-1)}\odot \cdots \odot U^{(1)})^{\trp}.
\end{equation}

\begin{figure}[tpb]
 \centering
 {\includegraphics[width=0.37\textwidth]{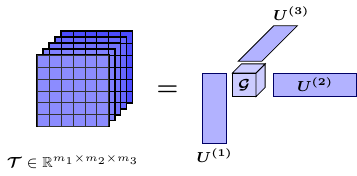}}
 {\includegraphics[width=0.57\textwidth]{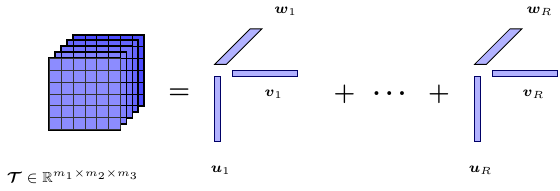}}
 \caption{Left: Tucker decomposition. Right: CP decomposition.}
 \label{fig:tucker-cpd}
\end{figure}

A point in the search space $\man=\prodsp$ of~\eqref{prog:main} is denoted by
$U=(U^{(1)},\dots,U^{(k)})$, where $U^{(i)}\in\reals^{m_i\times R}$ is the $i$-th {\it factor matrix} of $U$. %
In fact, $\man$ is a smooth manifold, and the tangent space to $\man$ at any point $U\in\man$ is $\tansp[U]\man=\prodsp$. %
Therefore, a tangent vector in $\tansp[U]\man$ is also a tuple of $k$ factor matrices, denoted as $\xi=(\xi^{(1)},\dots,\xi^{(k)})$, where $\xi^{(i)}\in\reals^{m_i\times R}$, for $i=1,\dots,k$. %
The Euclidean metric at $U$ is defined and denoted as, for all $\xi,\eta\in\tansp[U]\man$, $\braket{\xi,\eta} = \sum_{i=1}^{k} \trace({\xi^{(i)}}^{\trp} \eta^{(i)})$, %
where $\trace(\cdot)$ is the trace of a matrix. 

Let $\man$ be endowed with a Riemannian metric $g$, then the Riemannian gradient of a real-valued smooth function $f$ at $U\in\man$, denoted as $\rgrad[U]$, is the unique element in $\tansp[U]\man$ that satisfies, for all $\xi\in\tansp[U]\man$, 
$g_{U}(\rgrad[U],\xi) = \dop f(U)[\xi]$, %
where $\dop f(U)$ denotes the \revj{(Euclidean)} first-order differential of $f$ at $U$. \revj{In
particular, the Euclidean gradient of $f$ at $U$ is denoted as $\nabla f(U)$.} %

\paragraph{Problem statement} 
By setting the regularizer as $\psi(U)= \frac{\lambda}{2}
\sum_{i=1}^{k}\frob[U^{(i)}]^{2}$, %
in a similar way as in Maximum-margin Matrix Factorization~\cite{Srebro2005}, we specify the polyadic decomposition-based model~\eqref{prog:main} as follows, 
\begin{equation}\label{pb:main}
\underset{U\in\man=\prodsp}{\text{minimize}}~~ f(U):=
f_{\Omega}(U) + 
\frac{\lambda}{2}
\sum_{i=1}^{k}\frob[U^{(i)}]^{2},
\end{equation}
where the first term $f_{\Omega}(U):= \frac{1}{2p}\fro{\po(\llbracket U^{(1)},\cdots, U^{(k)}\rrbracket- \tstar)}^{2}$ %
is referred to as the data fitting function of the problem, and $p=|\Omega|/(m_1\cdots m_k)$ is a constant called the {\it sampling rate}. %
Note that when the regularization parameter $\lambda>0$, the
regularizer $\psi$ has an effect of keeping the variable $U$ in a compact subset of $\man$.

\section{Algorithms}\label{sec:algorithm}
In this section, we propose a new metric on the manifold $\man=\prodsp$. 
Under the proposed metric, we develop Riemannian gradient descent and Riemannian conjugate gradient algorithms on $\man$.

\subsection{A preconditioned metric}\label{ssec:varmetric}
The Riemannian preconditioned algorithms in~\cite{Mishra2012a,kasai2016low} on
the product space of full-column rank matrices were proposed to
improve the Euclidean gradient descent method. 
In these algorithms, a Riemannian metric %
was defined on the search space according to the differential properties of the
cost function. Specifically, it is designed based on an
operator that approximates the ``diagonal
blocks'' of the second-order differential of the cost function. The chosen
metric plays a role of preconditioning in the optimization algorithms such as
Riemannian gradient descent. We refer to~\cite{mishra2016riemannian} for a more
general view on this topic of Riemannian preconditioning. 

Starting from the idea of Riemannian preconditioning, we design a metric for the polyadic
decomposition-based problem~\eqref{prog:main}, in which the rank constraints are more relaxed than the fixed Tucker-rank constraint in~\cite{kasai2016low}. %
First, we construct an operator
$\mathcal{H}(U)\revj{:\tansp[U]\man\mapsto\tansp[U]\man}$ using the ``diagonal blocks'' of the second-order differential of the data fitting function $f_{\Omega}$ in~\eqref{pb:main}. %
Roughly speaking, such an operator satisfies
\begin{equation}\label{eq:diff2-approx}
    \braket{\mathcal{H}(U)[\xi],\eta} \approx \nabla^{2}f_{\Omega}(U)[\xi,\eta],
\end{equation}
for all $\xi,\eta\in\tansp[U]\man$, where $\nabla^{2} f_{\Omega}(U)$ denotes the
\revj{(Euclidean)} second-order differential of $f_{\Omega}$, and $\braket{\cdot,\cdot}$ is 
the Euclidean metric. 
\revv{Then, we design a metric that behaves
locally like $(\xi,\eta)\mapsto\braket{\mathcal{H}(U)[\xi],\eta}$. %
Assume that the operator $\mathcal{H}(U)$
is invertible and that $\tilde{g}:
(\xi,\eta)\mapsto\braket{\mathcal{H}(U)[\xi],\eta}$ forms a Riemannian metric on
$\man$, then the Riemannian gradient $\grad[f_{\Omega}](U)$ satisfies, by definition, %
$\tilde{g}_{U}(\grad[f_{\Omega}](U),\xi) = \dop f_{\Omega}(U)[\xi]= \braket{\nabla f_{\Omega}(U),\xi}$ for all $\xi\in\tansp[U]\man$, 
where $\nabla f_{\Omega}(U)$ denotes the Euclidean gradient of $f_{\Omega}$. Hence, 
$\grad[f_{\Omega}](U)={\mathcal{H}(U)}^{-1}\left[\nabla f_{\Omega}(U)\right]$.
In view of $\mathcal{H}(U)$ as in~\eqref{eq:diff2-approx}, $\grad[f_{\Omega}](U)$ is an approximation of the Newton direction of $f_{\Omega}$, %
which results in a much improved convergence behavior than the Euclidean
gradient descent.}

\revj{For the tensor completion problem~\eqref{prog:main}, one difficulty in designing such a metric using $\mathcal{H}(U)$ as
in~\eqref{eq:diff2-approx} is that we need to find an appropriate approximation
to the second-order differential $\nabla^2 f_{\Omega}(U)$.} 
\revj{This requires finding the explicit forms of the first and second-order
derivatives of $f_{\Omega}$ as defined in~\eqref{pb:main}, and hence those of the CPD map $\Psi: \man\mapsto \reals^{m_1\times\dots\times m_k}:U\mapsto \cpdp[U]$.} 
\revj{Concretely, we proceed as follows.}

First, we deduce the second-order partial derivatives of $f_{\Omega}$. %
Let $U$ be a point in $\man$. %
The \revj{Euclidean gradient of $f_{\Omega}$ at $U\in\man$ is} 
\begin{equation}\label{eq:def-egrad}
  \nabla f_{\Omega}(U) =(\partial_{U^{(1)}} f_{\Omega}(U),\dots,\partial_{U^{(k)}} f_{\Omega}(U)),
\end{equation}
where the partial derivatives have the following element-wise expression, for $i=1,\dots, k$, %
\begin{equation}\label{eq:fpartui}
    \left[\partial_{U^{(i)}} f_{\Omega}(U)\right]_{\ell,r}:=\left.{\ddt f_{\Omega}(U^{(1)},\dots,U^{(i)}+t E_{\ell,r},\dots,U^{(k)})}\right|_{t=0}, %
\end{equation} 
where %
$E_{\ell,r} = \mathbf{e}_{\ell}\mathbf{e}_{r}^{\trp}$. %
The right-hand side of~\eqref{eq:diff2-approx} can be written in terms of the following second-order partial derivatives, %
\[
\nabla^2 f_{\Omega}(U)[\xi,\eta] := \sum_{i=1}^{k} \braket{\partial^{2}_{ii} f_{\Omega}(U)[\xi],\eta} + 
\sum_{i=1,j'\neq i}^{k} \braket{\partial^{2}_{ij'}f_{\Omega}(U)[\xi],\eta},\]
where 
\begin{equation}\label{eq:def-fpartij}
\partial^{2}_{ij}f_{\Omega}(U)[\xi]:= \left.{\ddt\left(\partial_{U^{(i)}}f_{\Omega}(U^{(1)},\dots,U^{(j)}+t \xi^{(j)},\dots,U^{(k)})\right)}\right|_{t=0},
\end{equation} 
for $i,j=1,\dots,k$.  
Subsequently, we construct $\mathcal{H}(U)$ as an operator on $\tansp[U]\man$ based
on the action of the aforementioned ``diagonal blocks'' of
$\nabla^{2}f_{\Omega}(U)$, \ie, $\partial^2_{ii}f_{\Omega}(U)[\xi]$ for
$i=1,\dots,k$. More accurately, 
to design a metric that works for all tensor completion problems with a certain subsampling pattern $\Omega$, %
we use the expectation of these terms over the subsampling operator. %
Therefore, we define $\mathcal{H}(U):\tansp[U]\man\mapsto\tansp[U]\man$ as follows,
\begin{equation}\label{eq:def-oph}
\mathcal{H}(U)[\xi] = \left( \expe[\Omega]\left[\partial^2_{11}f_{\Omega}(U)[\xi]\right],\dots,\expe[\Omega]\left[\partial^2_{kk}f_{\Omega}(U)[\xi]\right]\right).
\end{equation}

Given the polyadic decomposition-based data fitting function $f_{\Omega}$
in~\eqref{pb:main}, $f_{\Omega}$ can be rewritten in terms of $U^{(i)}$ and $\ujexi$ through the mode-$i$ tensor matricizations~\eqref{eq:def-cpd-mat}: %
$f_{\Omega}(U) = \frac{1}{2p}\frobn{\po[\Omega_{(i)}]\left(U^{(i)}(\ujexi)^{\trp} - \tstar_{(i)}\right)}^{2}$, 
where $\Omega_{(i)}$ is the mode-$i$ matricization of $\Omega$. %
Therefore, from~\eqref{eq:fpartui}, the first-order derivatives have the following expression, %
\begin{equation}\label{eq:fparti}
    \partial_{U^{(i)}} f_{\Omega} (U) =\frac{1}{p} \mathcal{S}_{(i)}\ujexi,
\end{equation}
where $\mathcal{S}_{(i)}$ is the mode-$i$ matricization of the residual %
$\mathcal{S}:= \po\left(\cpdp-\tstar\right)$. %
Combining~\eqref{eq:def-fpartij} and~\eqref{eq:fparti}, it follows that %
\begin{align}
    \partial^{2}_{ii} f_{\Omega} (U)[\xi] = \frac{1}{p}\po[\Omega_{(i)}]\left(\xi^{(i)}(\ujexi)^{\trp}\right)\ujexi.\label{eq:fpartii}
\end{align} %
Let $\Omega$ be a random index set of entries that are i.i.d.\ samples of the Bernoulli distribution: %
$(i_{1},\dots,i_{k})\in\Omega$, with probability $p$, for $(i_1,\dots,i_k)\in\irange[m_1]\times\dots\times\irange[m_k]$. 
Then by taking the expectation over the index set $\Omega$,%
~\eqref{eq:fpartii} has the following approximation 
\begin{equation}\label{Hessian_diag_approx}
\mathbb{E}_{\Omega}\left[\partial^{2}_{ii} f_{\Omega}(U) [\xi]\right] = \xi^{(i)} (\ujexi)^{\trp}\ujexi. 
\end{equation}
Therefore, the operator $\mathcal{H}(U)$ as defined in~\eqref{eq:def-oph} reads   
\begin{equation}\label{eq:def-hu-tc}
\mathcal{H}(U)[\xi]=\left(\xi^{(1)} (U^{(j)})^{\odot_{j\neq 1}})^{\trp}(U^{(j)})^{\odot_{j\neq 1}},\dots, \xi^{(k)} (U^{(j)})^{\odot_{j\neq k}})^{\trp}(U^{(j)})^{\odot_{j\neq k}}\right).
\end{equation}
\revj{Note that, by using the approximation~\eqref{Hessian_diag_approx}, we finally obtain an operator $\mc{H}(U)$ that is independent of the subsampling index set $\Omega$.} %
\revj{Interestingly, the idea of Riemannian preconditioning using~\eqref{eq:diff2-approx}--\eqref{eq:def-hu-tc}, for the problem~\eqref{pb:main}, is similar to the Riemannian Gauss--Newton (RGN) method for minimizing $\frac{1}{2p}\fro{\Psi(U)-\tstar}^2$ in~\cite{breiding2018riemannian}, where the normal equation of their RGN method involves the Jacobian $J$ of the CPD map $\Psi:U\mapsto\cpdp[U]$. The operator~\eqref{eq:def-hu-tc} is thus similar to the term $J^T J$ in their RGN method.}

A second difficulty is that the operator $\mathcal{H}(U)$ in~\eqref{eq:def-hu-tc} may not be invertible, since 
the $R$-by-$R$ symmetric matrices in the form of $(\ujexi)^{\trp}\ujexi$ are not necessarily positive definite. To this end, we propose to regularize $\mathcal{H}(U)$ with the identity operator on $\tansp[U]\man$. This is done by {\it shifting} the aforementioned matrices by adding a constant diagonal matrix $\delta I_{R}$, where $I_{R}$ denotes the $R$-by-$R$ identity matrix and $\delta>0$ is a relatively small parameter. 
Consequently, we define the following inner product. %

\begin{definition}\label{def:metric}
Given $U=(U^{(1)},\dots,U^{(k)})\in\man$, let $g_{U}$ 
be an inner product in $\tansp[U]\man$ as follows, 
\begin{equation}\label{metric}
	g_{U}(\xi,\eta)=\sum_{i=1}^{k}\trace\left( \xi^{(i)}\hmati(\eta^{(i)})^{T} \right),\text{~for~} \xi, \eta\in\tansp[U]\man, 
\end{equation}
where $\hmati$ is a $R$-by-$R$ matrix defined as 
\begin{equation}\label{def:mat-precon}
\hmati=(\ujexi)^{\trp}\ujexi + \delta I_{R}
\end{equation} and $\delta>0$ is a constant parameter. 
\end{definition}

Note that $\hmati$ is positive definite even when there is a rank-deficient factor matrix among $\{U^{(i)}\}$. 
Moreover, $g$ is smooth on $\man$, and is therefore a Riemannian metric.

Now, we consider $\man$ 
as a manifold endowed with the Riemannian metric~\eqref{metric}. The associated norm of a tangent vector $\xi\in\ttansp$ is defined and denoted by $\|\xi\|_{U}=\sqrt{g_U(\xi, \xi)}$. 
Based on the Euclidean gradient of $f$ in~\eqref{eq:def-egrad}, the Riemannian gradient of $f$ is, by definition, 
\begin{equation}\label{grad}
	\text{grad}f(U) = \left(\partial_{U^{(1)}}f(U) \hmati[1]^{-1},\cdots,\partial_{U^{(k)}}f(U) \hmati[k]^{-1} \right).
\end{equation}
As mentioned in the beginning of this subsection, the Riemannian gradient~\eqref{grad} can be seen as the result of preconditioning of the Euclidean gradient $\nabla f(U)$ with the operator $\mathcal{H}(U)$ (upto a ``rescaling'' with $\delta I_{R}$). %
We refer to the new metric~\eqref{metric} as the {\it preconditioned metric} on $\man$.

\subsection{The Riemannian preconditioned algorithms}\label{ssec:algs}

With the gradient defined in~\eqref{grad} using Riemannian preconditioning, we adapt Riemannian gradient descent and Riemannian conjugate gradient algorithms (\eg,~\cite{AbsMahSep2008}) %
to solve the problem~\eqref{pb:main}. 

\begin{algorithm}[htbp]
\caption{Riemannian Gradient Descent (\textsc{RGD})}\label{alg:generic-rgd}
\begin{algorithmic}[1]
\REQUIRE{$f:\mtotal\mapsto\reals$,  $x_{0}\in\mtotal$, tolerance $\epsilon > 0$; $t=0$.}
\ENSURE{$x_t$.} 
\WHILE{$\|\trgrad[x_t]\|>\epsilon$}
\STATE{Set $\ttanv[\eta]_t=-\trgrad[x_t]$.\hfill\# See~\eqref{grad}}
\STATE{Set stepsize $s_t$ through one of the rules~\eqref{def:ls-linemin},~\eqref{eq:ls-armijo} or~\eqref{eq:stepsize_bb}.} 
\STATE{Update: $x_{t+1} = x_t + s_t\ttanv[\eta]_t$; $t\leftarrow t+1$.}
\ENDWHILE{}
\end{algorithmic}
\end{algorithm}

The Riemannian gradient algorithm is given in Algorithm~\ref{alg:generic-rgd}, which consists mainly of setting the descent direction as the negative Riemannian gradient and selecting the stepsizes. Note that since the search space is $\man=\prodsp$, the retraction map in this algorithm (line 4) is chosen as the identity map. %
\revv{In line~3, given an iterate $x_{t}\in\man$ and $\eta_{t}\in\tansp[x_{t}]\man$, the stepsize $s_t$ is chosen by one of the following three methods.} %

\paragraph{Stepsize by line minimization} The line minimization consists in computing a stepsize as follows, %
\begin{equation}\label{def:ls-linemin}
  s_t=\argmin_{s>0} h(s):= f(x_{t}+s\eta_t). 
\end{equation} 
With third-order tensors ($k=3$), the solution can be obtained numerically by selecting %
from the roots of the derivative $h'(s)$, 
which is a polynomial of degree $5$.  

\paragraph{Backtracking line search with the Armijo condition}
We first set up a trial stepsize~$s_t^{0}$ using
the classical strategy~\cite[\S3.4]{nocedal2006numerical}: %
(i) when $t\leq 1$, $s_{t}^{0} = 1$, (ii) when $t\geq 2$,
$s_{t}^{0} = 2(f(x_{t-1})- f(x_{t-2}))/g_{x_{t-1}}(\eta_{t-1}, \trgrad[x_{t-1}])$; 
then the stepsize $s_t$ is returned by a backtracking procedure, \ie,
finding the smallest integer $\ell\geq 0$ such that 
 \begin{equation}\label{eq:ls-armijo}
     f(x_{t}) - f\left(x_{t} + s_t \ttanv[\eta]_{t}\right) \geq \sigma s_t
     g_{x_{t}}(-\trgrad[x_{t}],\ttanv[\eta]_{t}), 
\end{equation}
for $s_t:= \max(s_t^{0}\beta^{\ell}, s_{\min})$ with a constant parameter $s_{\min}>0$. 
The backtracking parameters are fixed with $\sigma, \beta\in(0,1)$. 

\paragraph{The Barzilai--Borwein (BB) stepsize}
Recently, the Riemannian Barzilai--Borwein (RBB) stepsize~\cite{iannazzo2018riemannian} has proven to be an efficient stepsize rule for Riemannian gradient methods. Hence, we choose the stepsize as %
\begin{equation}\label{eq:stepsize_bb}
{s_t^{\mathrm{RBB}1}} := \frac{\|z_{t-1}\|_{x_{t}}^{2}}{| g_{x_{t}}(z_{t-1}, y_{t-1}) |}, \quad\mathrm{or} \quad
{s_t^{\mathrm{RBB}2}}  :=\frac{| g_{x_{t}}(z_{t-1},y_{t-1}) |}{\|y_{t-1}\|_{x_{t}}^{2}},
\end{equation}
where $z_{t-1} = x_{t} - x_{t-1}$ and $y_{t-1} = \trgrad[x_{t}] - \trgrad[x_{t-1}]$.

The Riemannian conjugate gradient (RCG) algorithm is 
similar to RGD (Algorithm~\ref{alg:generic-rgd}) in terms of %
the stepsize selection (line 3) and the update step (line 4), but differs with RGD in the choice of the search direction (line 2). %
More specifically, the search direction of RCG is defined as %
\begin{equation*}
\eta_{t} = -\rgrad[x_{t}] + \beta_t \eta_{t-1},  
\end{equation*}
where $\beta_t$ is the CG parameter. %
In the numerical experiments, we choose the Riemannian version~\cite{manopt} of the modified Hestenes--Stiefel rule~\cite{Hestenes&Stiefel:1952} (HS+) as follows, 
\begin{equation*}
     \beta_t = \max\left(0, \frac{g_{x_{t}}(\xi_{t}-\xi_{t-1},
                    \xi_{t})}{g_{x^t}(\xi_{t}-\xi_{t-1},\eta_{t-1})}\right). \label{eq:rncg-beta-hs} 
\end{equation*}
The vector transport operation involved in the computation of
$\beta_{t}$ is chosen to be the identity map. 

In both algorithms, the cost for calculating the Riemannian gradient~\eqref{grad} is a dominant term of the total cost. 
Therefore, we propose an efficient method for evaluating the Riemannian gradient in the next section.

\subsection{Computation of the gradient}\label{ssec:comp-grad}
We focus on the computation of the Riemannian gradient of $f$ in~\eqref{pb:main}. 
From the definition~\eqref{grad}, 
the computation of %
$\text{grad}f(U)=(\eta^{(1)},\dots,\eta^{(k)})$ consists of two parts: (i)
computing the partial derivatives $D_i:=\partial_{U^{(i)}}
f(U)\in\reals^{m_i\times R}$ and (ii) computing the matrix multiplications
$\eta^{(i)}=D_i \hmati^{-1}$. 
From the expressions~\eqref{eq:fparti} and~\eqref{def:mat-precon}, we have
\begin{align}
  D_i&= \frac{1}{p}\underbrace{\mathcal{S}_{(i)} (U^{(j)})^{\odot_{j\neq i}}}_{\tilde{D}_{i}}+\lambda U^{(i)},\label{eq:def-di-tdi}\\
    \eta^{(i)} &= D_i(\underbrace{(\ujexi)^{\trp}\ujexi+\delta
    I_{R}}_{\hmati})^{-1}.\label{eq:def-etai}
\end{align}
\revv{In a straightforward manner,} these two parts require mainly the following operations:
\begin{enumerate}
    \item Computing the sparse tensor $\mathcal{S}$ as
        in~\eqref{eq:fparti}, 
        which requires $2|\Omega|R$ flops;
    \item Computing $\ujexi$ for $i=1,\dots,k$, which requires 
        $2\sum_{i=1}^{k}m_{-i}R$ flops, where $m_{-i}:= \prod_{j\neq i} m_j$; 
    \item Forming the sparse matricizations $\mathcal{S}_{(i)}$ for $i=1,\dots,k$,
        which takes some extra time for input/output with the sparse tensor
        $\mathcal{S}$ and the matricizations of the index set $\Omega$. 
    \item Computing the sparse-dense matrix products
        $\tilde{D}_i:=\mathcal{S}_{(i)}\ujexi$, for $i=1,\dots,k$, which require $2k|\Omega|R$ flops. Then, one has access to $\{D_i\}_{i=1,\dots,k}$ after the matrix additions with $\lambda U^{(i)}$ (whose cost is not listed out since it is fixed regardless of the computational method). 
    \item Computing the $R$-by-$R$ matrix $\hmati$~\eqref{def:mat-precon} based on the matrix $\ujexi$ (obtained in step 2), which consists of a dense-dense matrix multiplication of sizes $R\times m_{-i}$ and $m_{-i}\times R$, for $i=1,\dots,k$, which mainly requires 
        $2\sum_{i=1}^{k}m_{-i} R^2$. 
    \item Computing $D_i\hmati^{-1}$ given $D_i$ (obtained in step 4) and $\hmati$ (obtained in step 5), through Cholesky
        decomposition of $\hmati$, for $i=1,\dots,k$, which requires
        $\sum_{i=1}^{k}2m_{i}R^2+\cchol R^3$. 
\end{enumerate}
The sum of the flops counted in the above list of operations is 
\begin{equation}\label{eq:cost-naive}
2(k+1)|\Omega|R + (\sum_{i=1}^{k} 2m_{-i}(R^2+R)) + (\sum_{i=1}^{k} 2m_{i}R^2 + \cchol R^3).
\end{equation}

\paragraph{An efficient computational method} We propose a computational method that avoids the matricizations of the residual tensor $\mathcal{S}$ and the computations of the Khatri--Rao products $\ujexi$. 

Given the residual tensor $\mathcal{S}$ after the step 1 above, we propose to compute $\tilde{D}_{i}$ in~\eqref{eq:def-di-tdi} without passing through the steps 2 and 3. 
In fact, the computation of $\tilde{D}_{i}=\mathcal{S}_{(i)} (U^{(j)})^{\odot_{j\neq i}}$ 
corresponds to the Matricized tensor times
Khatri--Rao product (MTTKRP), which is a common routine in the tensor computations. %
Through basic tensor computations, the entrywise expression of this MTTKRP %
does not require forming the matricizations %
of $\mathcal{S}$ explicitly. 
For brevity, we demonstrate these relations concretely in the case of third-order tensors ($k=3$), knowing that their extension to higher-order tensors is straightforward. 
The matrix $\tilde{D}_1 = \mathcal{S}_{(1)} (U^{(j)})^{\odot_{j\neq 1}}\in\reals^{m_1\times R}$ in~\eqref{eq:def-di-tdi} 
has the entrywise expression below, %
\begin{equation}\label{eq:mttkrp-ent}
    \left[\tilde{D}_{1}\right]_{i_1\ell} =  \revv{\sum_{i_2=1}^{m_2}\sum_{i_3=1}^{m_3}} \mathcal{S}_{i_1 i_2 i_3} U^{(3)}_{i_3\ell} U^{(2)}_{i_2\ell},
\end{equation}
for $(i_1,i_2,i_3)\in\irange[m_1]\times\dots\times\irange[m_3]$ and
$\ell=1,\dots,R$. 
Based on~\eqref{eq:mttkrp-ent}, Algorithm~\ref{algo:mttkrp} presents an efficient way to compute the MTTKRPs of~\eqref{eq:def-di-tdi}, with a sparse residual tensor $\mathcal{S}$ as input. Note that 
the computations of $\tilde{D}_{2}$ and $\tilde{D}_{3}$ correspond to the same equation~\eqref{eq:mttkrp-ent} but with the indices $(i_1,i_2,i_3)$ swapped via the rotations $(1,2,3;2,3,1)$ and $(1,2,3; 3,1,2)$ respectively. 

\begin{algorithm}[!htbp]
\caption{Sparse MTTKRP}\label{algo:mttkrp}
\begin{algorithmic}[1]
\REQUIRE{The index sets by axis $I_{\Omega}:=\{i: (i,j,k)\in\Omega\}$, $J_{\Omega}$ and $K_{\Omega}$. The sparse tensor $\mathcal{S}$ in the form of a $|\Omega|$-by-$1$ array of observed entries $\{S_{p}=\mathcal{S}_{i_{p},j_{p},k_{p}}: (i_{p},j_{p},k_{p})\in\Omega\}$. The factor matrices $(U^{(i)})_{i=1,2,3}$.}
\ENSURE{$\tilde{D}:= \mathcal{S}_{(1)}U^{(3)}\odot U^{(2)}$.} 
\STATE $\tilde{D} = \mathbf{0}$. 
\FOR{$p=1,\dots,|\Omega|$}
\FOR{$\ell=1,\dots,R$}
\STATE $\tilde{D}_{i_{p}\ell} =  \tilde{D}_{i_{p}\ell} + \mathcal{S}_{p} U^{(3)}_{k_{p}\ell} U^{(2)}_{j_{p}\ell}$. 
\ENDFOR
\ENDFOR
\end{algorithmic}
\end{algorithm}

Subsequently, we propose to compute $\hmati$~\eqref{def:mat-precon} without
large matrix multiplications. In fact, the large matrix multiplications with
$\ujexi$ in~\eqref{def:mat-precon} can be decomposed into smaller ones. Note
that these matrix multiplications satisfy the following identity:
\begin{equation}\label{eq:comp-khatri-rao-prod}
  \left( (U^{(j)})^{\odot_{j\neq
    i}}\right)^{\trp}\left((U^{(j)})^{\odot_{j\neq i}}\right) = G_k\star\cdots
    G_{i+1}\star G_{i-1}\star\cdots\star G_1,
\end{equation} 
where $G_j := {U^{(j)}}^{\trp}U^{(j)}$ for $j\neq i$ and the product by $\star$
denotes the Hadamard product. %
Using this property, the computation of $\hmati$ reduces to computing
$G_j={U^{(j)}}^{\trp}U^{(j)}$, and then the entrywise
multiplications between the (small) $R$-by-$R$ matrices $\{G_j\}$, 
which require only $\sum_{i=1}^{k}(2m_i+k-1)R^2$ flops, since the computation of
the matrices $G_j$ cost $2\sum_{i=1}^{k}m_iR^2$ and the entrywise
multiplications between $G_j$ cost $(k-1)R^2$. 

\revv{In summary, the operations reduce to the following steps.}
\begin{enumerate}
    \item[a.] Computing the sparse tensor $\mathcal{S}$ as
        in~\eqref{eq:fparti}. 
        This is identical to the step 1 above, which requires $2|\Omega|R$ flops; 
    \item[b.] Computing $\tilde{D}_i:=\mathcal{S}_{(i)}\ujexi$, for
        $i=1,\dots,k$, using $\mathcal{S}$ (obtained in step a) and $U$; see
        Algorithm~\ref{algo:mttkrp}. The computational cost of this step is
        $2k|\Omega|R$. Then, one has access to $\{D_i\}_{i=1,\dots,k}$ after the matrix additions with $\lambda U^{(i)}$. 
    \item[c.] Computing the $R$-by-$R$ matrix $\hmati$~\eqref{def:mat-precon}
        using $U$ (input data); see~\eqref{eq:comp-khatri-rao-prod}. The
        computational cost of this step is $\sum_{i=1}^{k} (2m_i+k-1)R^2$. 
    \item[d.] Computing $D_i\hmati^{-1}$ given $D_i$ (obtained in step b) and $\hmati$ (obtained in step c), through Cholesky
        decomposition of $\hmati$, for $i=1,\dots,k$. This is identical to the step 6 above, which requires $\sum_{i=1}^{k}2m_{i}R^2+\cchol R^3$
        flops.
\end{enumerate}
Therefore, the total cost of the above steps is
\[
2(k+1)|\Omega|R + \sum_{i=1}^{k} (4m_i +k-1)R^2  + \cchol R^3,
\]
which is significantly reduced compared to the cost~\eqref{eq:cost-naive} of the naive method. 
In particular, for 
third-order (or a bit higher-order) tensors ($k\ll m_i$) with a low rank parameter $R$, %
the dominant term in~\eqref{eq:cost-naive} %
is $2(k+1)|\Omega|R + \sum_{i=1}^{k} (2m_{-i}+2m_i)R^2$, 
while the dominant term in the cost of the proposed method is $2(k+1)|\Omega|R + \sum_{i=1}^{k} 4m_{i}R^2$. 
The reduction in the cost can be seen from the fact that $m_i\ll m_{-i}=\prod_{j\neq i}m_j$ and $m_i\ll |\Omega|=p m_1\dots m_k$.
Note that on top of the above reduction in flops, the time efficiency is further improved as the matricizations of the residual tensor $\mathcal{S}$ are not needed. %
In particular, speedups related to the step b (instead of steps 2--4 in the naive method) are demonstrated in Table~\ref{tab:comp-egrad-matlabvsmex}. %

\section{Convergence analysis}\label{sec:convergence} 
In this section, we analyze the    
convergence behavior of 
Algorithm~\ref{alg:generic-rgd}. Let $\algseq$ denote the sequence generated by this algorithm. %
First, we demonstrate in Proposition~\ref{prop:conv-limpt} that %
every accumulation point %
of $\algseq$ is a stationary point. Second, we analyze the iterate convergence property of the algorithm in Theorem~\ref{thm:main-rgd}.  

The following lemma generalizes the class of functions with Lipschitz-continuous gradient to functions defined on Riemannian manifolds, and will be used in Lemma~\ref{lemm:suff-desc}. 
\begin{lemma}[{\cite[Lemma~2.7]{Boumal2018}}]\label{lemm:lipschitz}
Let $\man'\subset\man$ be a compact Riemannian submanifold. Let
$\retr[x]:\tansp[x]\man'\mapsto\man'$. If
$f:\man'\mapsto\reals$ has Lipschitz continuous gradient in the convex
hull of $\sls$. Then there exists $L > 0$ such that, for all
$x\in\man'$ and $\xi\in\tansp[x]\man'$, 
\begin{equation}
    \label{eq:f-lipschitz}
    |f(\retr[x](\xi)) - (f(x)+ g_{x}(\xi,\rgrad[x])| \leq \frac{L}{2}
    \|\xi\|^{2}_{x}.
\end{equation}
\end{lemma}

Starting from the above lemma, we show that the proposed algorithm ensures a sufficient decrease property at each iteration. 

\begin{lemma}\label{lemm:suff-desc}
Let $\algseq$ be the sequence generated by Algorithm~\ref{alg:generic-rgd} (RGD), in which the step sizes are chosen by either line minimization~\eqref{def:ls-linemin} or Armijo line search~\eqref{eq:ls-armijo}. For all $t\geq 0$, there exists $\bar{C}>0$ such that %
\begin{equation}
    \label{eq:suff-desc}
    f(x_{t}) - f(x_{t+1}) \geq \bar{C} \|\rgrad[x_{t}]\|^{2}_{x_{t}}.
\end{equation}
In particular, with the step sizes chosen by line minimization~\eqref{def:ls-linemin}, there exists a Lipschitz-like constant $L_0>0$ such that~\eqref{eq:suff-desc} holds for $\bar{C} = \frac{1}{2L_0}$. %
\end{lemma}
\begin{proof}
Due to the fact that the objective function is coercive (because of the
Frobenius norm-based terms), the sublevel set %
$\sls=\{x\in\man: f(x) \leq f(x_{0}) \}$ %
is a closed and bounded subset of $\man$. %
Due to the boundedness of $\sls$, the convex hull of $\sls$, denoted as $\bar{\sls}$, is bounded. Therefore, $f$ has Lipschitz continuous gradient in $\bar{\sls}$ since $f\in C^2(\man)$. From Lemma~\ref{lemm:lipschitz}, there exists a Lipschitz-like constant $L_0>0$ such that~\eqref{eq:f-lipschitz} holds. The inequality~\eqref{eq:f-lipschitz} ensures an upper bound of $f(\mathcal{R}_{x}(s\xi))=f(x+s\xi)$ as follows, %
$f(x+s\xi) \leq f(x) +g_{x}(s\xi,\rgrad[x]) + \frac{L_0}{2} \|s\xi\|^2_{x}$, %
for all $s\geq0$. %
Consequently, when the stepsize $s_t=s^*$ is selected by line minimization~\eqref{def:ls-linemin}, we have 
\begin{align*}
    f(x_{t+1})&=f(x_{t}-s^*\rgrad[x_{t}]) \\
    &\leq \min_{s\geq 0} \left(f(x) - s(1-\frac{L_0 s}{2}) \|\rgrad[x_{t}]\|^2_{x_{t}}\right) = f(x_{t})-\bar{C}\|\rgrad[x_{t}]\|_{x_{t}}^2, 
\end{align*}
with $\bar{C}=\frac{1}{2L_0}$. 
When the stepsize $s_t$ is selected using Armijo line search, %
the new iterate $x_{t+1}=x_{t}-s_t\rgrad[x_{t}]$ is an Armijo point, where $s_t\geq s_{\min}>0$, by construction of the line search procedure (with the parameter of lower bound of stepsizes $s_{\min}$). Hence, through~\eqref{eq:ls-armijo}, we have 
\begin{align*}
f(x_t)-f(x_{t+1}) &\geq \sigma s_t \|\rgrad[x_t]\|_{x_t}^2\geq\bar{C} \|\rgrad[x_t]\|_{x_t}^2,
\end{align*}
where $\bar{C}=\sigma s_{\min}>0$, with the line search parameter $\sigma\in (0,1)$. 

In conclusion, the sufficient decrease property is satisfied with the two
stepsize selection methods in the statement.
\end{proof}

\begin{proposition}\label{prop:conv-limpt}
The sequence $\algseq$ generated by Algorithm~\ref{alg:generic-rgd}, with step
sizes chosen by either line minimization~\eqref{def:ls-linemin} or 
Armijo line search~\eqref{eq:ls-armijo}, satisfies the following convergence properties: 
(i) Every accumulation point is a stationary point; 
(ii) The algorithm needs at most $\ceil{\frac{(f^{*}-
f(x_0))}{\bar{C}}\frac{1}{\epsilon^2}}$ iterations to reach an
$\epsilon$-stationary solution, for a constant $\bar{C} >0$. 
\end{proposition}
\begin{proof}
(i) Let $x_{*}\in\man$ be an accumulation point, then there exists a subsequence $(x_{k(t)})_{t\geq 0}$, where $\{ k(t): t\geq 0 \}\subset \mathbb{N}$, such that $\lim_{t\to\infty}(f(x_{k(t)})- f(x_{*})) = 0$. This entails that %
$\sum_{t=0}^{\infty} f(x_{k(t)}) - f(x_{k(t+1)}) = f(x_{k(0)})-f(x_{*}) < \infty$. 
  Applying the sufficient decrease property~\eqref{eq:suff-desc} of Lemma~\ref{lemm:suff-desc} to this previous inequality, we have 
  \begin{equation}\label{eq:ineq1-prop4.3}
  \sum_{t=0}^{\infty} \bar{C}\|\trgrad[x_{k(t)}]\|^2_{x_{k(t)}}\leq \sum_{t=0}^{\infty} f(x_{k(t)}) - f(x_{k(t+1)}) <\infty,
  \end{equation}
  for a constant $\bar{C}>0$. %
Therefore, $\lim\limits_{t\to\infty}\|\trgrad[x_{k(t)}]\|_{x_{k(t)}}=0$.  
(ii) Suppose that the algorithm does not attain an $\epsilon$-stationary point (a point on which the gradient norm is bounded by $\epsilon$) at iteration $T-1$, then $\|\trgrad[x_{t}]\|>\epsilon$, for all $0\leq t\leq T-1$. Using~\eqref{eq:suff-desc}, %
we have %
      $f(x_{0})- f(x_{T}) \geq \bar{C}\sum_{t=0}^{T-1}\|\trgrad[x_{t}]\|^{2}_{x_{t}}
      \geq \bar{C}\epsilon^{2}T$. %
  Therefore $T\leq \frac{f(x_{0})-f(x_{*})}{\bar{C}}\frac{1}{\epsilon^{2}}$. 
\end{proof}

Next, we prove the iterate convergence of the RGD algorithm in Theorem~\ref{thm:main-rgd} using the {\L}ojasiewicz property. 
We first give the definition of the {\L}ojasiewicz inequality for functions defined on a Riemannian manifold~\cite{schneider2015convergence}. %
\begin{definition}[{\L}ojasiewicz inequality~{\cite[Definition 2.1]{schneider2015convergence}}]\label{def:loj-ineq}
    Let $\mathcal{M}\subset\reals^{n}$ be a Riemannian submanifold of
    $\reals^{n}$. The function $f:\mathcal{M}\mapsto\reals$ satisfies a
    {\L}ojasiewicz gradient inequality at a point $x\in\mathcal{M}$, if
    there exists $\delta>0$, $\sigma>0$ and $\theta\in(0,1/2]$ such that
    for all $y\in\mathcal{M}$ with $\|y-x\|\leq \delta$, it holds that
    \begin{equation}
        \label{eq:loj-ineq}
        | f(x) - f(y) |^{1-\theta} \leq \sigma \| \rgrad[y]\|,
    \end{equation}
    where $\theta$ is called the {\L}ojasiewicz exponent. 
\end{definition}

The Proposition 2.2 of~\cite{schneider2015convergence} guarantees that~\eqref{eq:loj-ineq} is satisfied for real analytic functions defined on an analytic manifold. Since the objective function of~\eqref{pb:main} is indeed real analytic and that the search space $\man$ is an analytic manifold, the {\L}ojacisiewicz inequality~\eqref{eq:loj-ineq} holds. Consequently, we have the following iterate convergence result.

\begin{theorem}\label{thm:main-rgd}
    Let $\algseq$ be the sequence generated by Algorithm~\ref{alg:generic-rgd} with stepsizes chosen by either line minimization~\eqref{def:ls-linemin} or Armijo line search~\eqref{eq:ls-armijo}. %
    Then $\algseq$ converges to a stationary point $x_{*}\in\man$.
    Moreover, the local convergence rate of $\algseq$ follows:
    \begin{equation*}
    \label{eq:loj-convrate}
    \|x_{t} - x_{*} \|\leq C \twopartdef{e^{-c t}}{\theta =
    1/2,}{t^{-\theta/(1-2\theta)}}{otherwise,}
    \end{equation*}
    with the {\L}ojacisiewicz exponent $\theta\in(0, 1/2]$ and constants $c>0$ and $C>0$.
\end{theorem}

\begin{proof}
The inequality~\eqref{eq:suff-desc} 
ensures that the sequence $\algseq$ is monotonically decreasing. Hence, the RGD algorithm satisfies the conditions in~\cite[Theorem~2.3]{schneider2015convergence}. %
More precisely, it follows from~\eqref{eq:suff-desc} of Lemma~\ref{lemm:suff-desc} that
\begin{align}
|f(x_{t+1}) - f(x_{t})| &\geq \bar{C}\|\rgrad[x_{t}]\|^2_{x_{t}} = (\bar{C}/s_{t})\|x_{t+1}-x_{t}\|_{x_{t}}\|\rgrad[x_{t}]\|_{x_{t}}\notag \\
&\geq \kappa_0 \|x_{t+1}-x_{t}\|_{x_{t}}\|\rgrad[x_{t}]\|_{x_{t}}, \label{eq:thm-ineq1} 
\end{align}
where $\kappa_0>0$, since with line minimization~\eqref{def:ls-linemin}, $s_t=s^*>0$, for all $t\geq 0$, is chosen from a finite number of numerical solutions; and with Armijo line search~\eqref{eq:ls-armijo}, $0<s_{\min}\leq s_t\leq s_t^{0}$, where $s_t^{0}>0$ is the initial stepsize before backtracking. %
In addition, the RGD update rule ensures that 
\begin{equation}\label{eq:thm-ineq2}
\|x_{t+1}-x_t\|_{x_{t}} =s_t\|\trgrad[x_t]\|_{x_t}\geq \kappa\|\trgrad[x_t]\|_{x_t}, 
\end{equation} %
for $\kappa>0$. %
The result of the theorem is obtained by 
combining~\eqref{eq:thm-ineq1},~\eqref{eq:thm-ineq2} and the {\L}ojasiewicz inequality~\eqref{eq:loj-ineq} and using~\cite[Theorem~2.3]{schneider2015convergence}. 
\end{proof}

\revj{As far as we know, the global convergence of the RGD algorithm using RBB stepsizes without line search is not known. %
However, one can apply the RBB stepsize as an initial trial stepsize to the backtracking line search procedure; %
and consequently, all the convergence results in this section can be proved.} Interested readers are referred to~\cite{iannazzo2018riemannian} for details.

\section{Experiments}\label{sec:numerical}
In this section, we carry out numerical experiments for tensor completion using
the proposed algorithms and several existing algorithms in the related work. 
Details of these algorithms are as follows. 

The proposed Algorithm~\ref{alg:generic-rgd} is labeled as
Precon~RGD and the proposed RCG algorithm is labeled as Precon~RCG.
Depending on the stepsize selection method, these
algorithms are labeled with a descriptor (i) \revj{line minimization} (linemin)
for the stepsize rule~\eqref{def:ls-linemin}, %
and (ii) Riemannian Barzilai--Borwein (RBB)
for~\eqref{eq:stepsize_bb}. 
\revj{We choose to restrict ourselves to linemin and RBB in our experiments, since they appear to show better performances
in practice and are easier to use than the Armijo line search rule; see Appendix~\ref{appssec:comp} %
for a detailed discussion.} 
 
Euclidean gradient descent (Euclidean~GD) and nonlinear conjugate gradient
(Euclidean~CG) algorithms refer to the algorithms using the Euclidean
gradient~\eqref{eq:def-egrad} 
in the definition of the search directions on $\man$. The
stepsize selection rules are the same as the proposed algorithms; these algorithms are implemented along with the proposed algorithms in the source code. 
INDAFAC is a damped Gauss--Newton method for CPD-based tensor completion proposed by Tomasi and Bro~\cite{tomasi2005parafac}. 
CP-WOPT~\cite{acar2011scalable} is a nonlinear conjugate gradient algorithm for
CPD-based tensor completion. %
AltMin~\cite{guan2020alternating} is an alternating minimization algorithm for CPD-based tensor completion, which uses the linear CG for each of the least squares subproblems. 

KM16 refers to a Riemannian optimization algorithm proposed by Kasai and Mishra~\cite{kasai2016low} for tensor completion with a fixed Tucker rank. 
The Riemannian gradient in this algorithm is defined under a metric selected through Riemannian preconditioning on the manifold corresponding to a (fixed-rank) Tucker decomposition. 
In this algorithm, the tensor candidate is represented by a tuple of factor matrices and a core tensor via the Tucker
decomposition. %
In our experiments on third-order tensors, %
this algorithm is labeled as KM16~$(r_1, r_2, r_3)$, according to the Tucker rank $(r_1,r_2,r_3)$ with which it is tested. 
Note that the dimension of the search space of KM16 is $\sum_{i=1}^{k} \left(m_i r_i - r_i^2\right) + \prod_{i=1}^{k} r_i$, which is different than the dimension of $\man$ (search space of the CPD/PD-based algorithms); 
In particular, the difference in these dimensions is marginal when $r_i\approx R$ for $i=1,2,3$, with $R\ll \min(m_1,\dots, m_k)$.

All the CPD/PD-based algorithms are initialized with a same randomly generated point on $\man$ and the Tucker decomposition-based algorithm (KM16) is initialized with a point such that its tensor representation is close enough to that of the initial point of the other algorithms; see Appendix~\ref{appsec:exp} for details.

All numerical experiments were performed on a workstation with 8-core Intel Core i7-4790 CPUs and 32GB of memory running Ubuntu~16.04 and MATLAB~R2019. 
The source code is available at \url{https://gitlab.com/shuyudong.x11/tcprecon/}. Implementations of the existing algorithms %
are also publicly available. 

\subsection{Synthetic data}\label{ssec:exp-syn}

\paragraph{\revv{Tensor model}}
We consider a low-rank tensor model that is composed of a low Tucker-rank
tensor and independent additive noises: with a given Tucker rank parameter 
$r^{\star}=(\rstar_1,\rstar_2,\rstar_3)$, we generate such a tensor $\tstar$
using the following procedure, 
\begin{equation}\label{def:datamodel0}
    \tstar = \truncr(\mathcal{T}) + \mathcal{E},
\end{equation}
where $T\in\reals^{m_1\times m_2\times m_3}$ is a third-order tensor composed
of i.i.d.\ Gaussian entries, that is, $T_{ijk}\sim\mathcal{N}(0,1)$, the
operator in the form of $\truncr[r](\cdot)$ is a Tucker-rank ($\rktc$)
truncation operator defined as \revv{the best Tucker rank-$r$} approximation of $\mathcal{T}$. 
The truncation $\truncr[r](\cdot)$ can be obtained using existing implementations that are
available in state-of-the-art tensor toolboxes (\eg, Tensor Toolbox~\cite{TTB_Software} and
Tensorlab~\cite{tensorlab}). Here we use the function \texttt{tucker\_als.m} in the MATLAB Tensor Toolbox. 
In the scenario of noiseless observations, $\mathcal{E}=0$; otherwise
$\mathcal{E}\in\reals^{m_1\times m_2\times m_3}$ contains independent noises such that 
$\mathcal{E}_{\ell_{1}\ell_{2}\ell_{3}}\sim\mathcal{N}(0,\sigma)$, where
$\sigma$ is set according to a given signal-to-noise ratio (SNR); see
Appendix~\ref{appsec:exp}. %

\paragraph{\revv{Low-rank tensor recovery from partial, noiseless observations}}
\label{ssec:synth-n0}

A synthetic tensor $\tstar$ is generated with the model~\eqref{def:datamodel0} without noise. 
The tensor $\tstar$ is only observed on an index set $\Omega$, which is composed of indices drawn from the Bernoulli distribution: 
$
(i,j,k)\in\Omega \text{~with probability~} p\in(0,1)$, 
for all $(i,j,k)\in \llbracket m_1\rrbracket\times\llbracket m_2\rrbracket\times\llbracket m_3\rrbracket$. 

For the problem model~\eqref{pb:main}, we set the regularization parameter $\lambda$ to zero, which allows for recovering the low-rank tensor $\tstar$ without any bias, provided that the sampling rate $p$ is sufficient. 
Then we test the aforementioned algorithms with a given rank parameter $R$, assuming that the rank (CP or Tucker rank) of the hidden tensor $\tstar$ is unknown to all the algorithms. 
For the CPD and PD-based algorithms, we set the rank parameter $R$ 
to an arbitrary value such that $R \geq \max(\rstar_1, \rstar_2,\rstar_3)$. Since 
the optimal CP rank of the tensor candidate 
is unknown, %
a larger-than-expected rank parameter is interesting because it allows for searching solutions in a fairly large tensor space, so that there is better chance that 
optimal solutions are in the search space~$\man$. %
\revj{For parameter $\delta$ of~\eqref{def:mat-precon} involved in our
Riemannian gradients, we choose to use very small values since we are mostly interested in the performance of the Riemannian preconditioning technique. In all the experiments of Section~\ref{sec:numerical}, we set $\delta=10^{-7}$.}

The termination of the
proposed algorithms (Precon RGD and Precon RCG) is controlled by a tolerance
parameter ($\epsilon=10^{-7}$ in this experiment) against the norm of the gradient; 
KM16 uses a Riemannian CG algorithm with Armijo line search %
and default stopping criteria. 
On top of their respective stopping
criteria, all the algorithms are tested within heuristic iteration budget
maxiter $=1000$ \revj{and $T_{\max}=100$ seconds}. %
Note that for all
tested algorithms except AltMin, one iteration corresponds to one pass over the
whole training data $P_{\Omega}(\tstar)$; for AltMin, one iteration corresponds to multiple
passes over the training data, since each of its iteration has a number of inner
iterations for solving the underlying alternating subproblem. %
Table~\ref{tab:synth-n0} shows the performances of the tested algorithms in
terms of recovery errors and time, under the \revj{sampling rate $p=0.3$ and} rank parameters $R\in\{12,14,16\}$. The iteration histories of these algorithms (including KM16) with $R=14$ are shown
in \cref{fig:n0rstar-rperi-allms}. 
Specifically, for the fixed-Tucker rank algorithm, KM16, we also test several other rank parameters than $r=(R,R,R)$;
its tensor recovery performances along with those of the proposed 
algorithms are presented in Table~\ref{tab2:synth-n0}. 
In Table~\ref{tab2:synth-n0}, ``\#variables'' indicates the dimension of the search space of each of
algorithms, depending on the rank parameters. 

\begin{figure}[htpb]
 \centering
 \subfigure[Training RMSEs]
 {\includegraphics[width=0.395\textwidth]{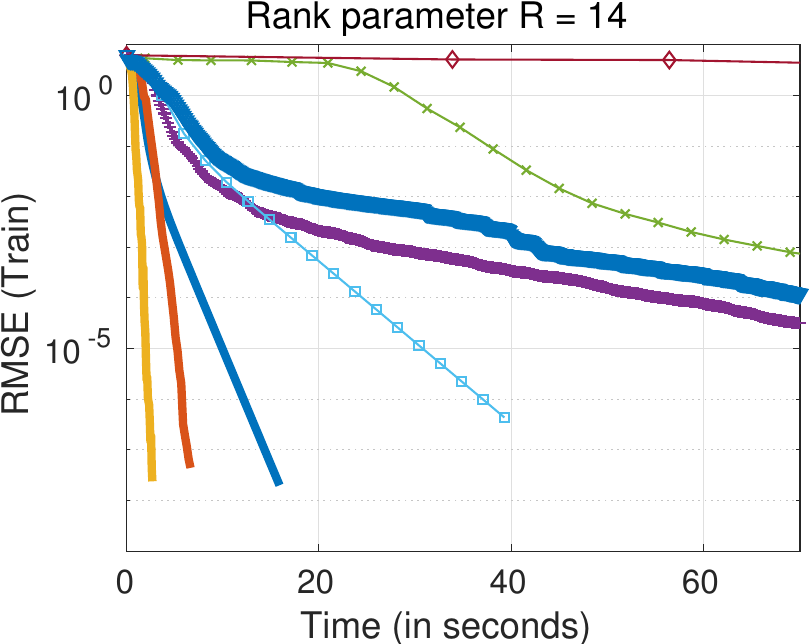}}
 \subfigure[Test RMSEs]
 {\includegraphics[width=0.38\textwidth]{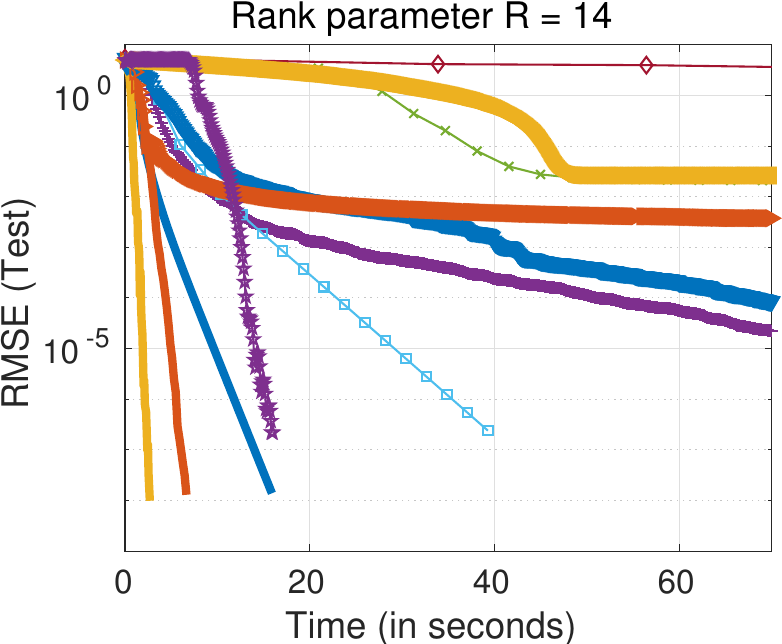}}
 \subfigure{\includegraphics[width=0.19\textwidth]{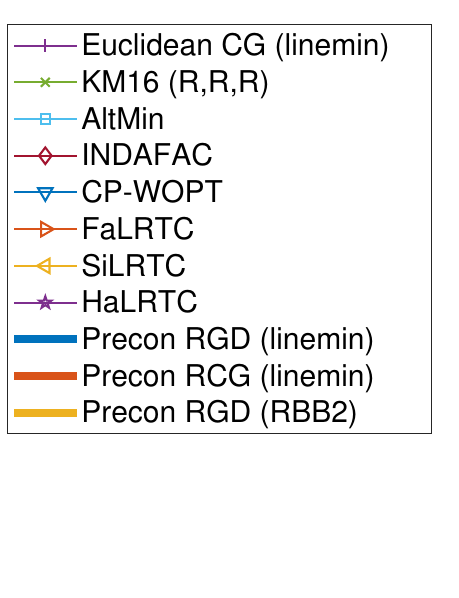}}
 \caption{Tensor completion from noiseless observations. 
The size of $\tstar$ is $(100, 100, 100)$ with a Tucker rank
$\rstar=(3, 5, 7)$. The sampling rate is $0.3$. The rank parameter is set as $\rkval=14$.}
\label{fig:n0rstar-rperi-allms}
\end{figure}

From the results shown in \cref{fig:n0rstar-rperi-allms} and Tables~\ref{tab:synth-n0}--\ref{tab2:synth-n0}, we have the following observations: 
(i) For all three values of the rank parameter $R$ that are larger than
$\max(\rstar_1,\rstar_2,\rstar_3)$, the proposed algorithms and HaLRTC~\cite{liu2012tensor} %
succeeded in recovering exactly the true hidden tensor $\tstar$ (with a test RMSE lower than $10^{-6}$). AltMin, CP-WOPT and
Euclidean~CG succeeded exact recoveries only with one or two of the rank parameter choices, and their convergences are slower than the proposed algorithms by orders of magnitude. %
The test error of KM16 stagnated at a certain level %
as the core tensor dimensions chosen are not exactly the same as the Tucker rank of $\tstar$; 
(ii) Among the algorithms that successfully recovered the true hidden
tensor, the
proposed algorithms (Precon~RGD and Precon~RCG) outperform AltMin %
in time with a speedup of around $10$ times, and they achieve 
speedups between $2$ and $6$ times compared to HaLRTC; Especially, Precon~RGD (RBB2) has the fastest convergence behavior, 
\revj{due to the fact that the RBB stepsize implicitly involves second-order information through a rough approximation of the Hessian.} 
(iii) For KM16 specifically, the time efficiency and the recovery
performance of KM16~$(r,r,r)$ 
improves significantly when the core tensor dimensions $(r,r,r)$ decrease
(and get closer to $\rstar$). In particular, when $r$ is only $1/2$ of the rank parameter $R=14$, the time efficiency of KM16~$(r,r,r)$
gets close to those of the proposed
algorithms. 
Note that when $r\approx R/2$, the dimensions of the search space of KM16~$(r,r,r)$ is much smaller than that of the proposed algorithms; see Table~\ref{tab2:synth-n0}. These
comparisons can be explained by the fact that the per-iteration cost of KM16 is
much larger than the proposed algorithms, even when the dimensions of its
search space is close or smaller than that of the proposed algorithms; see
\cref{fig:gdatn0rgeqstar_tperi} for detailed comparisons of their average
per-iteration time. 
\revv{The high computational cost of KM16 lies in its
computation of the Riemannian gradient, which scales poorly with the Tucker rank
as it involves computing the Gram matrix of the matricizations of the core
tensor---with a cost of $O(r_1 r_2r_3(r_1+r_2+r_3))$---and solving Lyapunov
equations (for $k=3$ times) of the $r_i\times r_i$ matrices.}

%
\begin{table}[htpb]
\footnotesize
\centering
\caption{Tensor completion with noiseless observations. The size of $\tstar$ is $(100, 100, 200)$ with a Tucker rank $\rstar=(3, 5, 7)$. 
The sampling rate is $0.3$. The rank parameters $\rkval$ tested are $\{12, 14, 16\}$. 
}
\label{tab:synth-n0}
\begin{tabular}{lrrrrr}
\hline\hline
Algorithm             & \multicolumn{1}{l}{$R$}  & \multicolumn{1}{l}{iter} & \multicolumn{1}{l}{time (s)} & \multicolumn{1}{l}{RMSE (test)} & \multicolumn{1}{l}{RMSE (train)}  \\
\hline                                                                       
Euclidean CG (linemin) & 12                       & 592                       & 100.03                   & 1.57e-12                    & 2.03e-12                      \\
AltMin                 & 12                       & 67                        & 100.97                   & 8.05e-06                    & 8.09e-06                      \\
INDAFAC                & 12                       & 7                         & 129.91                   & 3.97e-01                    & 4.57e-01                      \\
CP-WOPT                & 12                       & 405                       & 29.96                    & 5.74e-07                    & 6.70e-07                      \\
HaLRTC                 & 12                       & 142                       & 15.81                    & 2.19e-07                    & --                      \\
Precon RGD (linemin)   & 12                       & 96                        & 16.25                    & 3.84e-08                    & 4.79e-08                      \\
Precon RCG (linemin)   & 12                       & 48                        & 8.23                     & 1.96e-08                    & 3.58e-08                      \\
Precon RGD (RBB2)      & 12                       & 65                        & \textbf{3.91}                     & 9.52e-09                    & 4.74e-07                      \\
\hline                                                                       
Euclidean CG (linemin) & 14                       & 518                       & 100.11                   & 2.20e-06                    & 3.10e-06                      \\
AltMin                 & 14                       & 19                        & 39.27                    & 2.39e-07                    & 4.35e-07                      \\
INDAFAC                & 14                       & 5                         & 125.77                   & 1.35e+00                    & 2.60e+00                      \\
CP-WOPT                & 14                       & 1561                      & 94.29                    & 2.40e-06                    & 3.35e-06                      \\
HaLRTC                 & 14                       & 142                       & 15.97                    & 2.19e-07                    & --                      \\
Precon RGD (linemin)   & 14                       & 82                        & 15.90                    & 1.38e-08                    & 1.99e-08                      \\
Precon RCG (linemin)   & 14                       & 34                        & 6.65                     & 1.29e-08                    & 4.39e-08                      \\
Precon RGD (RBB2)      & 14                       & 39                        & \textbf{2.69}                     & 9.84e-09                    & 2.38e-08                      \\
\hline                                                                       
Euclidean CG (linemin) & 16                       & 456                       & 100.01                   & 1.37e-07                    & 1.53e-07                      \\
AltMin                 & 16                       & 8                         & 31.79                    & 2.67e-08                    & 3.32e-08                      \\
INDAFAC                & 16                       & 5                         & 154.34                   & 9.78e-01                    & 1.19e+00                      \\
CP-WOPT                & 16                       & 1766                      & 99.02                    & 5.23e-06                    & 7.00e-06                      \\
HaLRTC                 & 16                       & 142                       & 16.34                    & 2.19e-07                    & --                     \\
Precon RGD (linemin)   & 16                       & 40                        & 8.82                     & 1.56e-09                    & 2.51e-09                      \\
Precon RCG (linemin)   & 16                       & 50                        & 11.09                    & 2.59e-09                    & 5.08e-09                      \\
Precon RGD (RBB2)      & 16                       & 39                        & \textbf{3.03}                     & 1.53e-10                    & 3.14e-10                      \\
\hline\hline
\end{tabular}
\end{table}

\begin{table}[htpb]
\footnotesize
\centering
\caption{Tensor completion with the noiseless observations. Proposed algorithms vs KM16 $(r,r,r)$ with different choices of $r$. %
The size of $\tstar$ is ($100$, $100$,
$200$), with a Tucker rank $\rstar=(3, 5, 7)$.}
\begin{tabular}{cccrrr}
\hline\hline
Algorithm            & $R$ & \#variables & iter & Time (sec.)   & RMSE (test)    \\
\hline                                   
KM16 (R, R, R)         & -- & 7756         & 29   & 102.20 & 2.70e-02     \\ 
KM16 (12, 12, 12)      & -- & 6096         & 30   & 88.32  & 9.00e-04    \\ 
KM16 (9, 9, 9)         & --  & 4086         & 21   & 29.63  & 2.70e-05   \\
KM16 (7, 7, 7)         & --  & 2996         & 22   & 14.39  & 2.99e-05   \\
\hline                                   
Precon RGD (linemin) & 14    & 5600         & 82   & 15.90  & 1.38e-08   \\
Precon RCG (linemin) & 14    & 5600         & 34   & 6.65   & 1.29e-08   \\
Precon RGD (RBB2)    & 14    & 5600         & 39   & \textbf{2.69}   & 9.84e-09   \\
\hline\hline
\end{tabular}
\label{tab2:synth-n0}
\end{table}

\begin{figure}[htpb]
 \centering
 \includegraphics[width=0.72\textwidth]{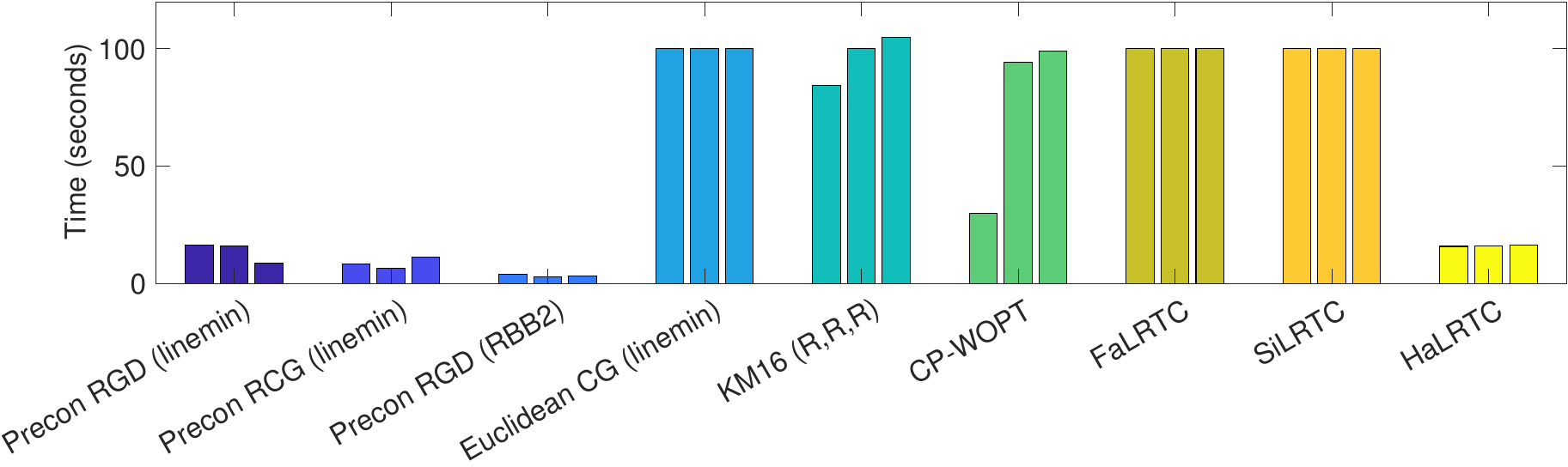}
 \\
 \hspace{1.4em}\includegraphics[width=0.84\textwidth]{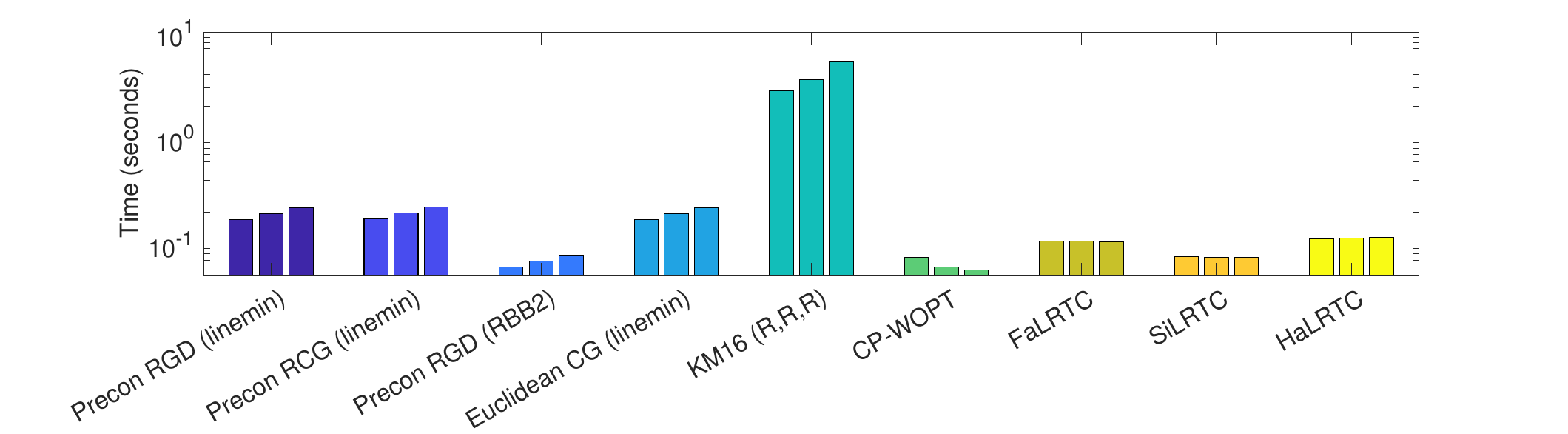}
\caption{\revj{Computation time for tensor completion tasks. (Top): total time
when reaching the target accuracy $\epsilon=10^{-7}$ or the time budget $T_{\max}=100$
        seconds; (Bottom):
        average time per iteration. The three bars of each algorithm correspond
        to the cases with $R=12$, $14$, and $16$ respectively. The size of
        $\tstar$ is $(100, 100, 200)$ with a Tucker rank $\rstar = (3, 5, 7)$.
        The sampling rate is $0.3$.}}
\label{fig:gdatn0rgeqstar_tperi}
\end{figure}

\revj{Furthermore, we make a more thorough test to evaluate the tensor completion performances of the proposed algorithms, alongside Euclidean CG for comparison, in the same tensor completion task on a $6\times 10$ grid of $(p, R)\in \{0.12,0.152,\dots,0.28\} \times \{4,6,\dots,16, 17, 19, 21\}$, for $5$ random runs at each $(p,R)$ (each run is based on a randomly generated index set $\Omega$); see results in \cref{fig:synth-n0-succr} in Appendix~\ref{appsec:exp-lrta}. These results provide a broader view on the difficulties---in terms of chances of exactly recovering $\tstar$---of the tensor completion task with different sampling rates and rank parameter values.}

\revj{Finally, in addition to tensor completion, we extend the application of the proposed algorithms to tensor approximation, where the tensor $\tstar$ is fully observed tensor and the goal is to approximate $\tstar$ with a tensor with an upper-bounded CP rank. We give performance comparisons between the proposed algorithms and the Riemannian Gauss--Newton method (RGN) of~\cite{breiding2018riemannian}; see details in Appendix~\ref{appsec:exp-lrta}.
}

\paragraph{\revv{Low-rank tensor recovery from partial, noisy observations}}\label{ssec:synth-noisy}
In the following experiments, we conduct tensor completion tests under the same
tensor model as in the previous experiment, except that the revealed
tensor entries are observed with additive noise, and the noise level $\sigma$
in the model~\eqref{def:datamodel0} is set according to a given
signal-to-noise ratio (SNR) of $40$ dB. \revj{To make this experiment similar to experiments on real data (MovieLens), in which case the sampling rate is usually at the order of $1\%$ or even lower, we set the sampling rate $p=5\%$.} 

In this experiment, the regularization parameter $\lambda$ of the problem~\eqref{pb:main} is selected from $\{0,1/p, 10^{1/2}/p, 10/p\}$ (where the scalar $p$ is the sampling rate) for a
rank parameter $R=14$ and the selected value of $\lambda$ is $10^{1/2}/p$. All the tested algorithms are terminated if the relative change of the training error attains a given tolerance value: %
\begin{equation}\label{def:relchg}
\texttt{relchg} = \frac{|E(U^{t+1})-E(U^{t})|}{|E(U^{t})|}\leq \texttt{tol}, 
\end{equation}
where $E$ denotes the training RMSE. 
Also, the algorithms terminates if a heuristic time budget $T_{\max}$ is attained. 
We set the tolerance parameter as $\texttt{tol}=10^{-6}$ %
and the time budget as $T_{\max}= 300$s (seconds).

\begin{figure}[htpb]
 \centering
 \subfigure[Training RMSEs of all algorithms]{\includegraphics[width=0.36\textwidth]{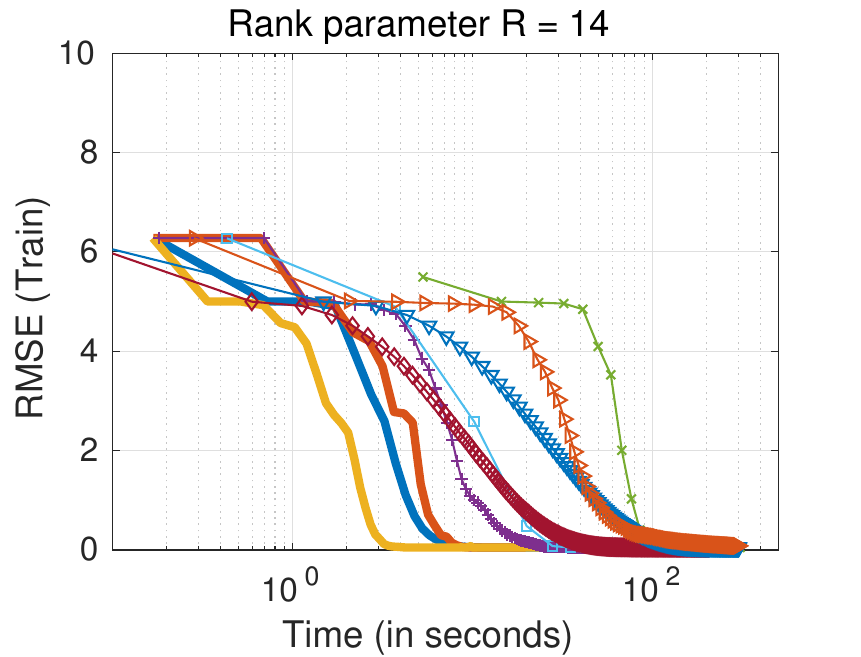}
 }
 \subfigure[Test RMSEs of all algorithms]{\includegraphics[width=0.36\textwidth]{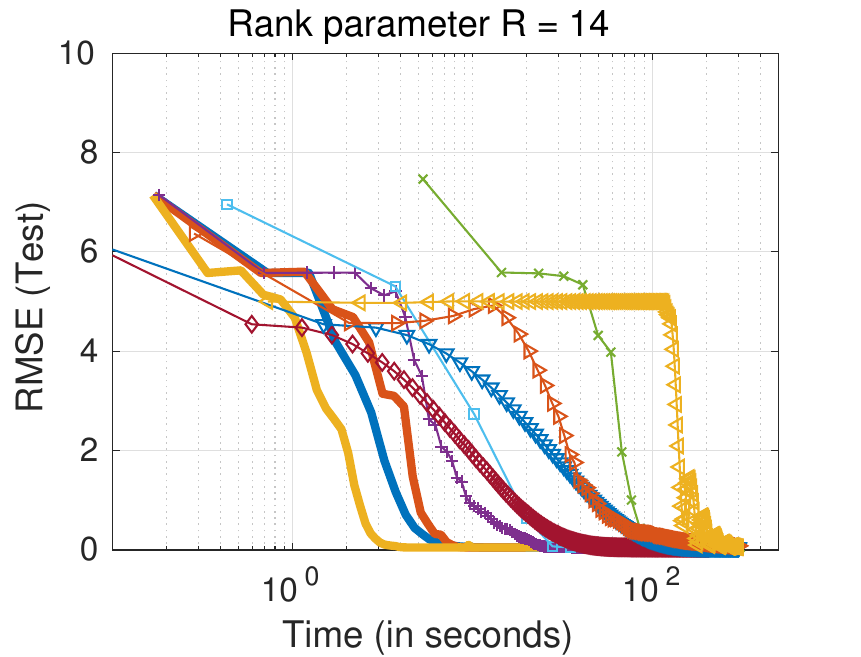}
 }
\subfigure{\includegraphics[width=0.19\textwidth]{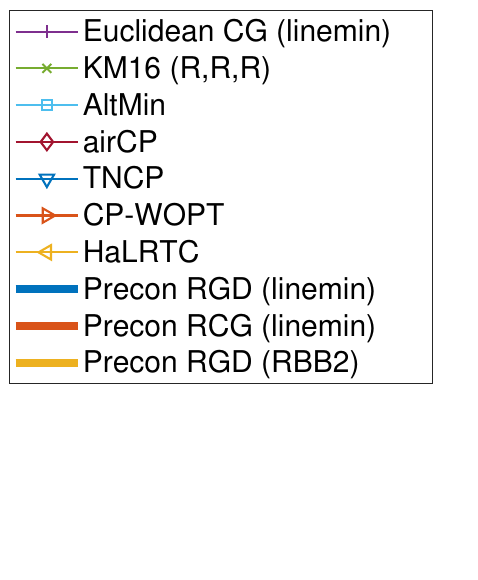}
 }\\
\subfigure[Comparison with KM16]{\includegraphics[width=0.36\textwidth]{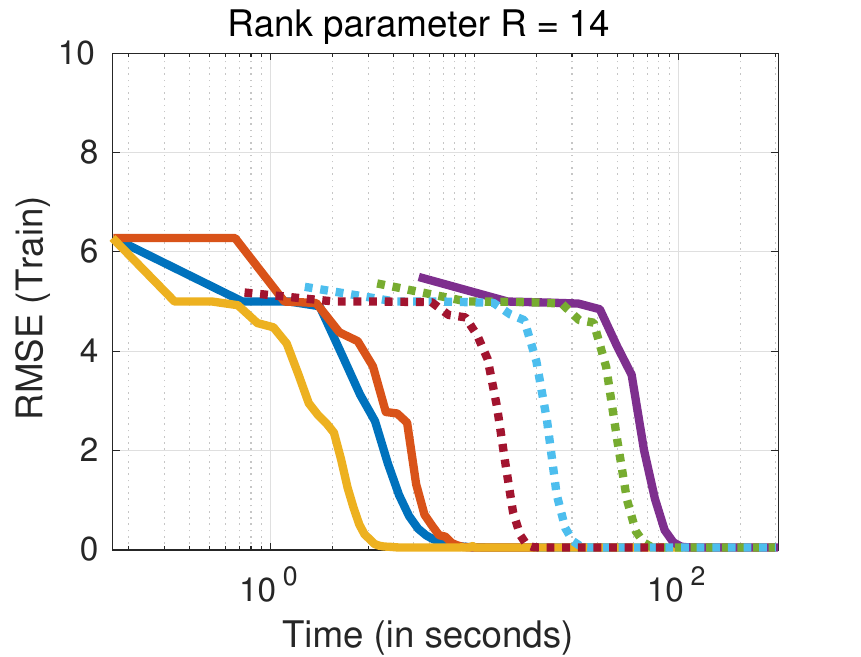}
 }
 \subfigure[Comparison with KM16]{\includegraphics[width=0.36\textwidth]{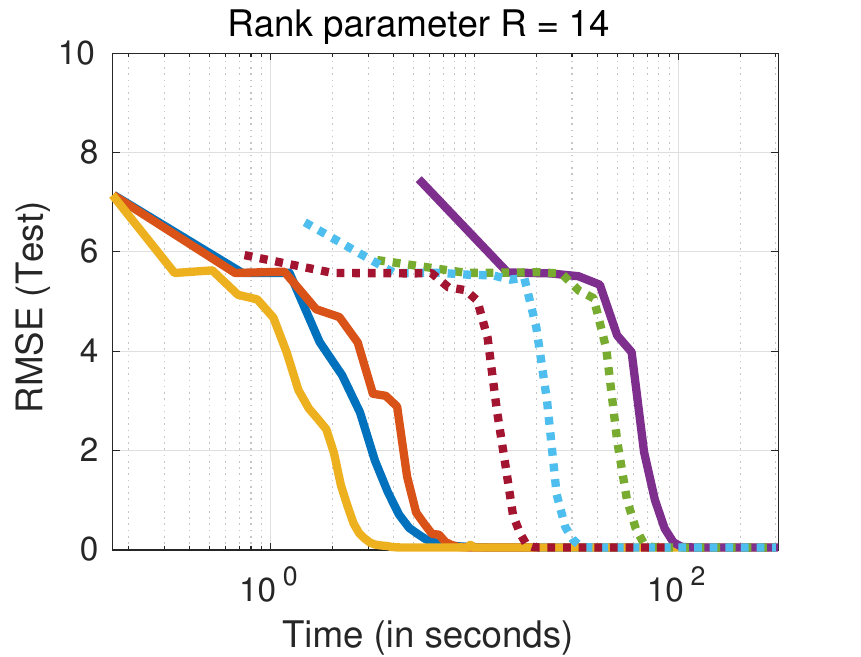}
 }
\subfigure{\includegraphics[width=0.19\textwidth]{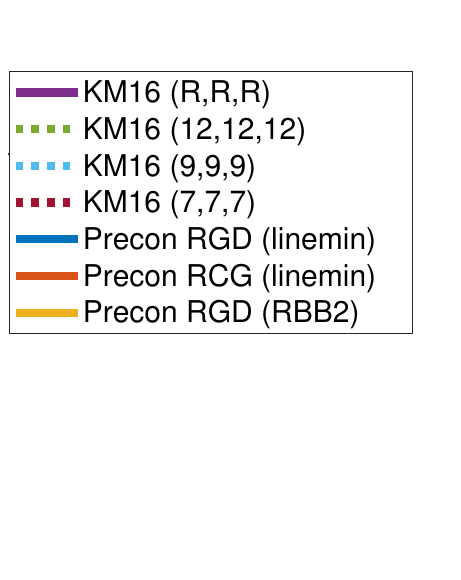}
 }
\caption{Tensor completion from noisy observations. The size of $\tstar$ is $(300, 500,
200)$ with rank $\rstar=(3, 7, 5)$. The sampling rate is $5\%$. The rank parameter $\rkval=14$.} 
\label{fig:ltcrk-noisy-rperi-allms}
\end{figure}

The performances of the tested algorithms are shown in
\cref{fig:ltcrk-noisy-rperi-allms}. 
From these results, we have similar observations as in the previous experiments: 
(i) For a polyadic decomposition rank $R$ that is larger than
$\max(\rstar_1,\rstar_2,\rstar_3)$, all algorithms achieve a
recovery performance of the same level (with a test RMSE around $0.050$); (ii)
the proposed algorithms (Precon~RGD and Precon~RCG) outperform the rest of the
algorithms in time with speedups between $4$ and $15$ times at the several
accuracy levels near their respective termination point; 
(iii) the time efficiency 
of KM16 $(r,r,r)$ improves significantly when the core tensor dimensions
$(r,r,r)$ decrease %
but it remains inferior to those of the proposed algorithms; see
\cref{fig:ltcrk-noisy-rperi-allms}(d)--\ref{fig:ltcrk-noisy-rperi-allms}(e).

\subsection{Real data}\label{ssec:exp-realdata}
In this subsection, we conduct experiments on a real-world dataset. 
We focus on evaluating the time efficiency of the proposed algorithms under various choices
of the rank parameter $R$. To ensure a good generalization performance of the
tensor completion model, we activate the regularization terms of the
problem~\eqref{pb:main} with a regularization parameter $\lambda>0$. 

\paragraph{Dataset and algorithms}
The tensor completion tests are conducted on the MovieLens~1M dataset\footnote{https://grouplens.org/datasets/movielens/1m/}, %
which consists of $1$ million movie ratings from $6040$ users on $3952$ movies and seven-month period from September 19th, 1997 through April 22nd, 1998. Each movie rating in this dataset has a
time stamp, which is the number of week during which a movie rating was given. Therefore, this dataset has a tensor form $\tstar$ of size $6040 \times
3952 \times 150$, where the first two indices are the user and movie identities
and the third order index is the time stamp. The dataset contains over $10^{6}$
ratings, which correspond to the known entries of the data tensor $\tstar$. 
For the tensor completion tasks, we randomly select $80\%$ of the known ratings
as the training set. Note that in this case, the absolute sampling rate
$p=|\Omega|/(m_1m_2m_3)$ is $2.23\%$. 

Due to the large dimensions of the data tensor of MovieLens~1M, several
aforementioned algorithms in the related work were not tested on this dataset
due to excessive memory requirements. In the implementation of these algorithms, %
the tensor index filtering operations such as accessing 
$\tstar_{|\Omega}$ or $\mathcal{T}_{|\Omega}$ in iterations, a dense tensor
format is used for the index set $\Omega$, which---for the same size as $\tstar$---requires the storage of over $3.5\times 10^{9}$ entries in total; Such a
data format poses a memory requirement bottleneck that blocks the test. 
On the other hand, the proposed algorithms and
Euclidean~GD/CG, KM16 and AltMin can be run without the memory issue
since they use the COO format for the training set $\Omega$,
which corresponds to only $10^{6}$ entries on the same dataset. Note that for all tensor decomposition-based algorithms,
the memory requirement for the decomposition variables (or factor matrices) is
$O((m_1+m_2+m_3)R)$---which is bounded by $10^5$ for any rank parameter
$R<100$ in the experiments on MovieLens-1M---is also memory-efficient. 
Therefore, the algorithms that are tested are: the proposed algorithms (Precon~RGD/RCG), Euclidean~GD/CG, KM16 and AltMin.

\paragraph{Experiments and results}
 
Given the data tensor $\tstar$ and the index set $\Omega$ as the training set, we conduct performance evaluations using various choices for the rank parameter $R$, after selection of the regularization parameter $\lambda$. 
The parameter $\lambda$ for the CPD-based algorithms %
is selected among $\{0, 1/p, 10^{1/3}/p, 10^{2/3}/p, 10/p, 10^{4/3}/p\}$ via $3$-fold cross validation using the Euclidean GD algorithm (instead of the proposed ones), 
where the rank parameter $R$ is set to be $5$, $10$ and $15$ respectively; and the values selected by these cross validation procedures are $10^{2/3}/p$ (when $R=5$ and $10$) and $10^{4/3}/p$ (when $R=15$). 
For the Tucker decomposition-based algorithm---KM16---with the Tucker rank $r=(R,R,R)$, for $R\in\{5, 10,15\}$, the values of $\lambda$ selected after the same cross validation procedure are $0$.
Subsequently, we test all algorithms using the selected parameters. 
Similar to previous tests, the stopping criteria for all the tested
algorithms use the relative change of training errors, \ie, \texttt{relchg} in~\eqref{def:relchg}, and a large enough time budget $T_{\max}$. 
We set the tolerance parameter for $\texttt{relchg}$~\eqref{def:relchg} as $\texttt{tol}=10^{-5}$ and the maximal time budget as $T_{\max}= 1800$s. 

We present the iteration histories of all algorithms under the rank parameter
$R=15$ in \cref{fig:ml1m-errs-ranks}(a). 
We also observe the recovery performances of the algorithms under a series of
different rank parameters---for $R\in\{1,\dots,15, 17,19\}$ and $\lambda =
10^{2/3}/p$; see \cref{fig:ml1m-errs-ranks}(b).

\begin{figure}[htpb]
 \centering
 \subfigure[Training RMSEs by time]{\includegraphics[width=0.42\textwidth]{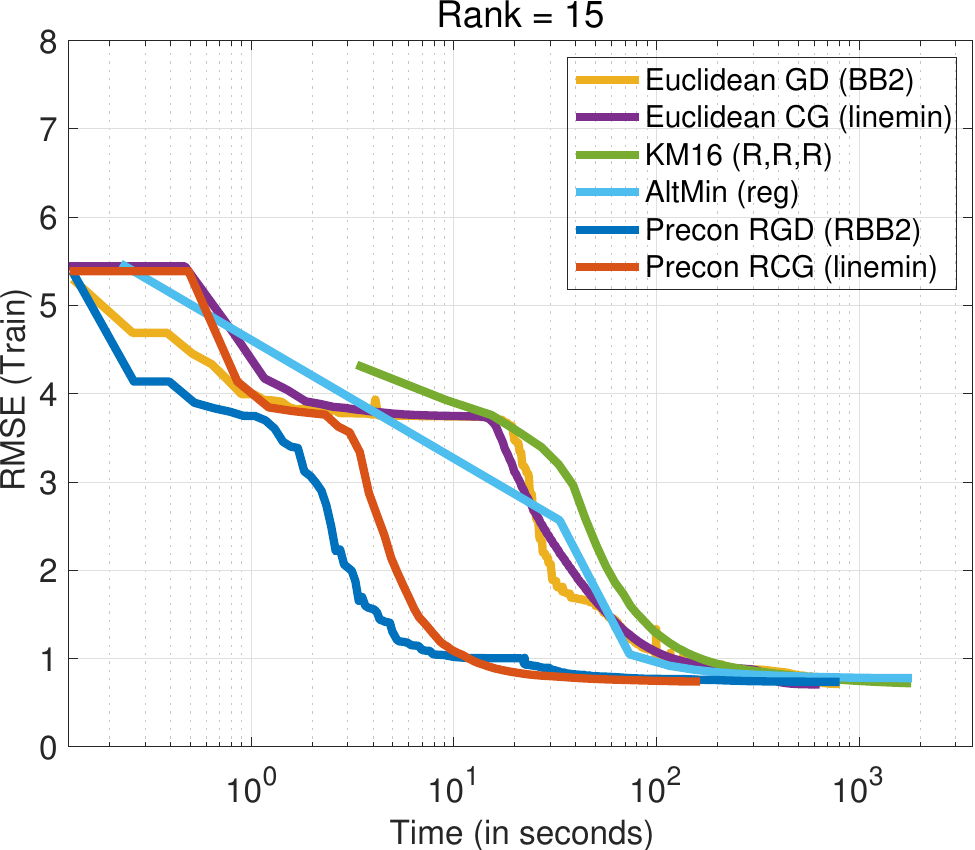}  
 }\qquad
\subfigure[Training and test RMSEs]{\includegraphics[width=0.50\textwidth]{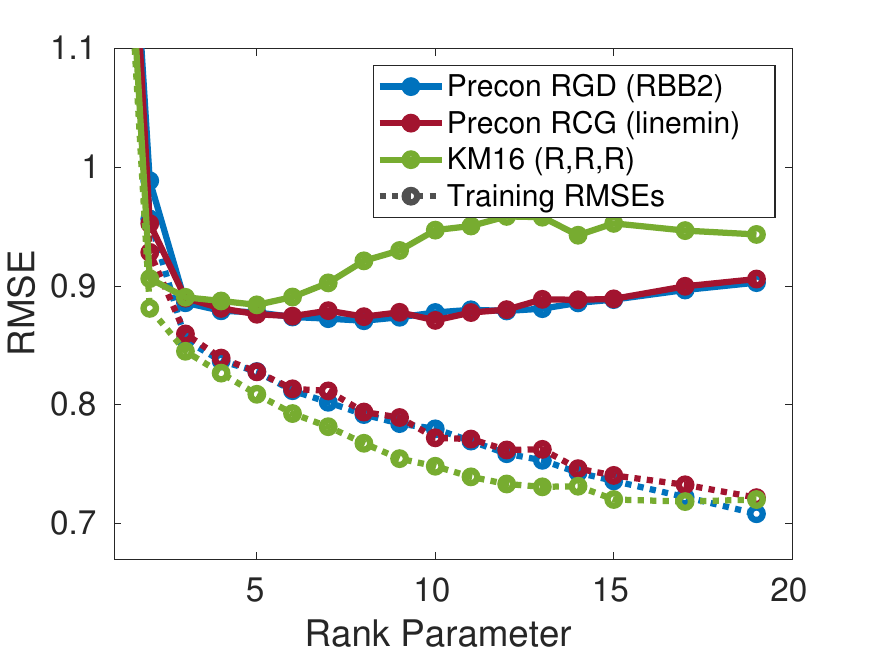}
 }
\caption{Tensor completion on the MovieLens~1M dataset. Recovery performances with various rank parameters and comparisons of the algorithms (with $R=15$) in time.}  
\label{fig:ml1m-errs-ranks}
\end{figure}

Moreover, \cref{fig:ml1m-rperi-multirks} shows the iteration histories of the proposed algorithm Precon~RGD (using the BB stepsize) in comparison with KM16, with the rank parameters $R\in\{3, 6, 15\}$. 

\begin{figure}[!htbp]
\centering
\subfigure[Training RMSE]{ 
\includegraphics[width=0.48\textwidth]{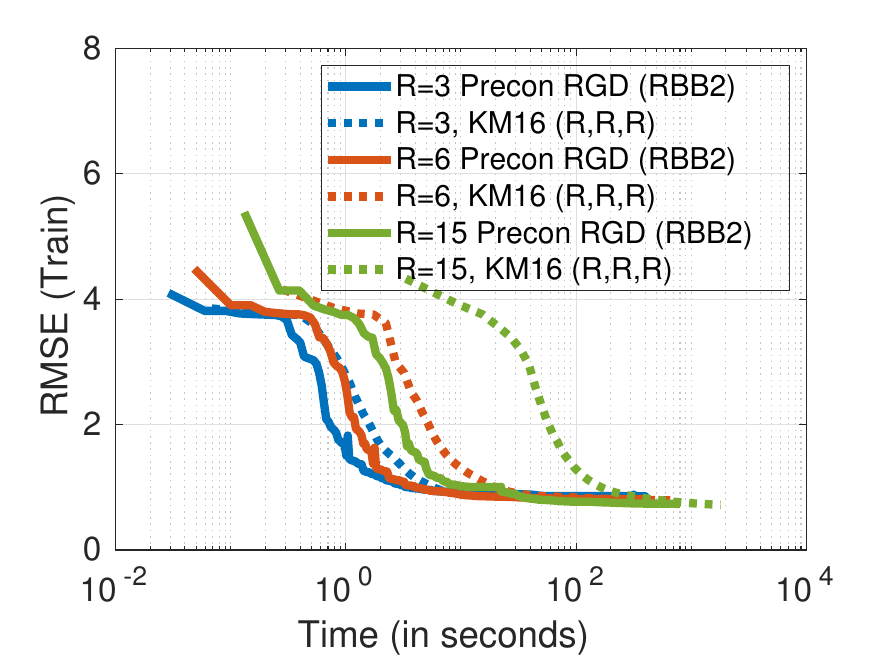}}
\subfigure[Test RMSE]{ 
\includegraphics[width=0.48\textwidth]{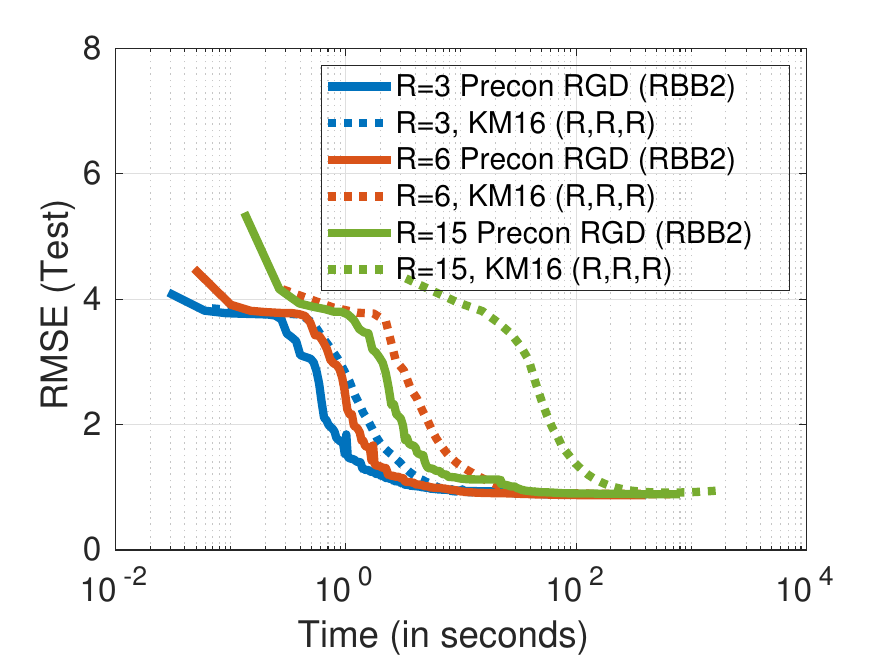}}
\caption{Proposed algorithms and KM16~\cite{kasai2016low}. Dataset: MovieLens~1M. The rank parameters $R$ are $\{3,6,15\}$.}
\label{fig:ml1m-rperi-multirks}
\end{figure}

From \cref{fig:ml1m-errs-ranks}(a), we
observe that (i) the two proposed algorithms have faster convergence behaviors than all the rest of the tested algorithms, and (ii) in particular, 
the proposed algorithms achieve the same recovery error with speedups of around $10$ times compared to KM16, and $6$ times compared to Euclidean~CG (linemin). 
The results in \cref{fig:ml1m-errs-ranks}(b) give an overview of the performance of our algorithms and KM16 in terms of their trade-offs between the model complexity (for the minimization of the fitting error) and the generalization of the model for the prediction of missing entries. From these results, we can see that (under the several randomly generated regularization parameters so-far explored), the rank choice of $R=8$ provides the best recovery error on (unknown) test entries. 
Moreover, for rank choices that are larger than $R=8$,
the decrease in the recovery performances of the proposed algorithms (compared to the case with $R=8$) is much smaller than that of KM16. This can be explained by the fact that the search space of our algorithms under larger rank parameters contain those with smaller ranks, while the search space of KM16 corresponds to matrix spaces of a fixed column-rank.

\section{Conclusion}\label{sec:conclusion}
We proposed a new class of Riemannian preconditioned first-order algorithms for tensor completion through low-rank polyadic decomposition. 
We have analyzed the convergence properties of Riemannian gradient
descent using the proposed Riemannian preconditioning. %
The main feature of the
proposed algorithms stems from a new Riemannian metric defined on the product space of
the factor matrices of polyadic decomposition. This metric induces a local
preconditioning on the Euclidean gradient descent direction of the PD-based
objective function; the underlying preconditioner has the form of an approximated
inverse of the diagonal blocks of the Hessian of the objective function. 

These Riemannian preconditioned algorithms %
share some characteristics with the related work~\cite{kasai2016low}, which deals with tensor completion with a fixed Tucker rank. They differ however in the following sense: the polyadic decomposition model allows for finding a low-rank tensor candidate within a range of CP ranks, %
while the algorithm of~\cite{kasai2016low} searches a tensor solution with a fixed (Tucker) rank.

Because of the more flexible decomposition modeling, our algorithms perform
well with various arbitrary choices of the rank parameter in the tensor
completion tasks on both synthetic and the MovieLens~1M datasets. %
Moreover, we have observed that the proposed algorithms provide significant
speedup over several state-of-the-art algorithms for CPD-based tensor completion, while providing comparable or better tensor recovery quality.  

\begin{appendices}

\section{Algorithmic details}\label{appsec:alg}

\paragraph{Speed-up by using C/MEX-based MTTKRP}\label{appssec:comp}
Algorithm~\ref{algo:mttkrp} is implemented in a mexfunction and has shown significant speedup over the so-far implemented computations (explicit sparse matricizations times the explicitly computed Khatri--Rao products). %
Table~\ref{tab:comp-egrad-matlabvsmex} shows comparative results on the MovieLens~1M dataset, ``Naive'' corresponds to the
implementation where the gradient computations involve 
(i) forming sparse matricizations of the residual tensor, (ii) computing the Khatri--Rao products and (iii)
sparse-dense matrix multiplication. ``Proposed'' corresponds to the results of the implementation using sparse MTTKRP
(Algorithm~\ref{algo:mttkrp}). %
\begin{table}[htpb]
\footnotesize
\centering
\caption{Speedups of the efficient computational method. The tensor dimensions are $6040\times
3952\times 150$.}
\begin{tabular}{c|rr|r|r}
\hline\hline
\multicolumn{1}{c|}{\multirow{2}{*}{iter}} & \multicolumn{2}{c|}{Time (s)}  &
\multicolumn{1}{c|}{Average} & \multicolumn{1}{c}{RMSE} \\ 
\multicolumn{1}{c|}{}     & \multicolumn{1}{c}{Naive} &
\multicolumn{1}{c|}{Proposed} & \multicolumn{1}{l|}{speedup} & \multicolumn{1}{l}{Naive/Proposed} 
\\ 
\hline
1         & 0.568                           & 0.067       & --                      & 4.778                     / 4.778                                            \\
101       & 71.848                          & 16.199      & 4.435$\times$           & 0.795                     / 0.795                                            \\
201       & 142.697                         & 32.333      & 4.413$\times$           & 0.765                     / 0.765                                            \\
301       & 213.897                         & 48.508      & 4.410$\times$           & 0.759                     / 0.759                                            \\
401       & 285.117                         & 64.609      & 4.413$\times$           & 0.759                     / 0.759                                            \\
501       & 356.405                         & 80.712      & 4.416$\times$           & 0.759                     / 0.759                                            \\
\hline\hline
\end{tabular}
\label{tab:comp-egrad-matlabvsmex}
\end{table}

\paragraph{Performance of the three stepsize methods}
We compare the performances of the three stepsize methods (in Section~\ref{ssec:algs}) in the same experimental settings as in Section~\ref{ssec:exp-syn}. %
\revj{The three stepsize methods for the proposed algorithms (Precon RGD and
RCG) are tested in tensor completion tasks with sufficiently large sampling
rate. %
In \cref{tab:synth-n0-comp-stepsz}, the time (s) and RMSE scores of each method correspond to the average of their respective values (at termination) after $20$ random tests.
\cref{tab:synth-n0-comp-stepsz} shows: (i) for Precon RGD, all three stepsize methods yield successful recovery results in terms of the average test RMSE but linemin and RBB2 are faster than Armijo; and (ii) for Precon RCG, the percentage of successful recoveries (\#Succ/20) with Armijo is inferior to those with linemin and RBB2. In particular, we show the iteration history of one failed attempt of Precon RCG (Armijo), in \cref{fig:gdatn0rgeqstar_comp_stepsz}(b), in which the algorithm stagnates prematurely. %
Therefore, for better practical performance in the experiment of Section~\ref{ssec:exp-syn}, we favor the choice of linemin and RBB2 for the computation of the stepsizes. 
}

\begin{table}[htpb]
\footnotesize
\centering
\caption{\revj{Comparison of the three stepsize meethods under the same setting
as in Section~\ref{ssec:exp-syn}. The time (s) and RMSE scores at termination
of each algorithm are the average values over the results of $20$ random tests.
The size of $\tstar$ is $(100, 100, 200)$ with a Tucker rank $\rstar=(3, 5,
7)$. The sampling rate is $0.3$. The Armijo rule uses default settings in~\cite{manopt}. } 
}
\label{tab:synth-n0-comp-stepsz}
\begin{tabular}{lrrrrrr}
\hline\hline
~ & Stepsize rule             & \multicolumn{1}{l}{$R$}  & \multicolumn{1}{r}{\#Succ/20} & \multicolumn{1}{r}{time (s)} & \multicolumn{1}{l}{RMSE (test)} & \multicolumn{1}{l}{RMSE (train)}  \\
\hline                                             
\multirow{3}{*}{Precon RGD}  & (Armijo)   & 14           & 20/20             & 12.176                      & 1.8932e-08                      & 1.8438e-08              \\ 
                             & (linemin)  & 14           & 20/20             & 6.169                       & 2.7023e-08                      & 3.2530e-08              \\
                             & (RBB2)     & 14           & 20/20             & 1.468                       & 2.5175e-08                      & 7.9817e-08              \\
\hline 
\multirow{2}{*}{Precon RCG}  & (Armijo)   & 14           & 18/20             & 2.630                       & 1.3037e-02                      & 1.3139e-02              \\
                             & (linemin)  & 14           & 20/20             & 3.919                       & 1.6501e-08                      & 2.9628e-08              \\
\hline\hline
\end{tabular}
\end{table}

\begin{figure}[htpb]
\centering
\subfigure[With Precon RGD]{\includegraphics[width=0.41\textwidth]{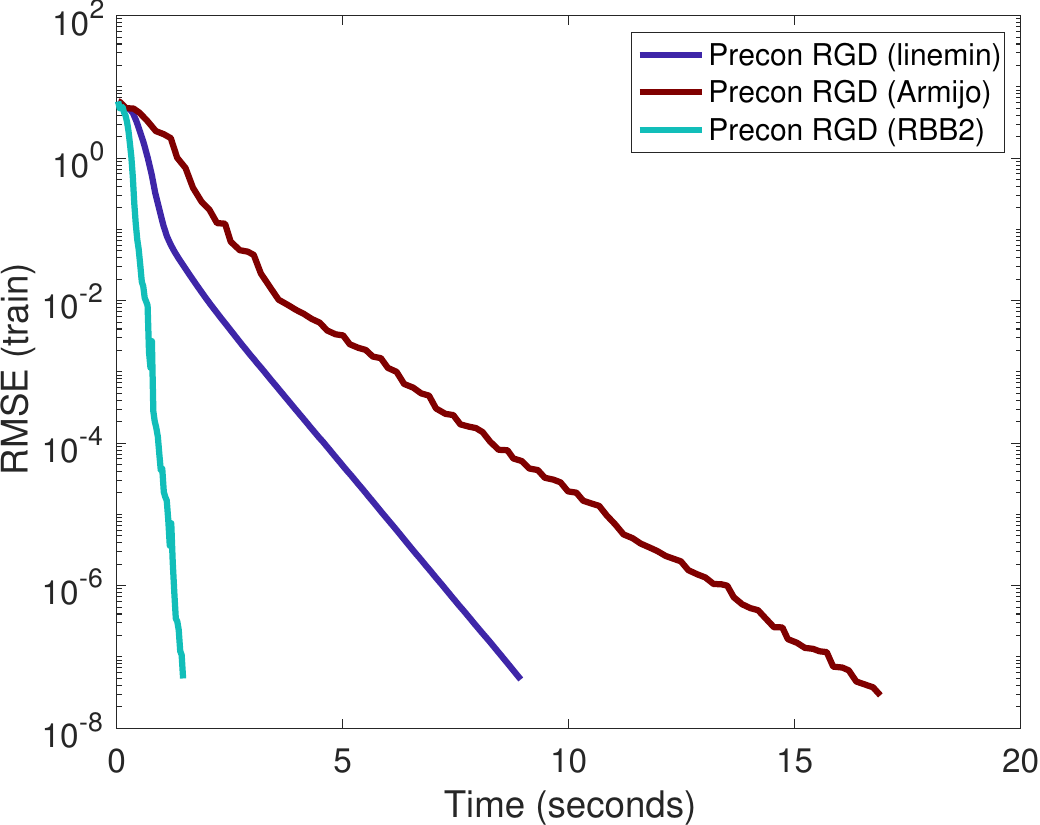}
 }
 \qquad%
 \subfigure[With Precon RCG]{\includegraphics[width=0.41\textwidth]{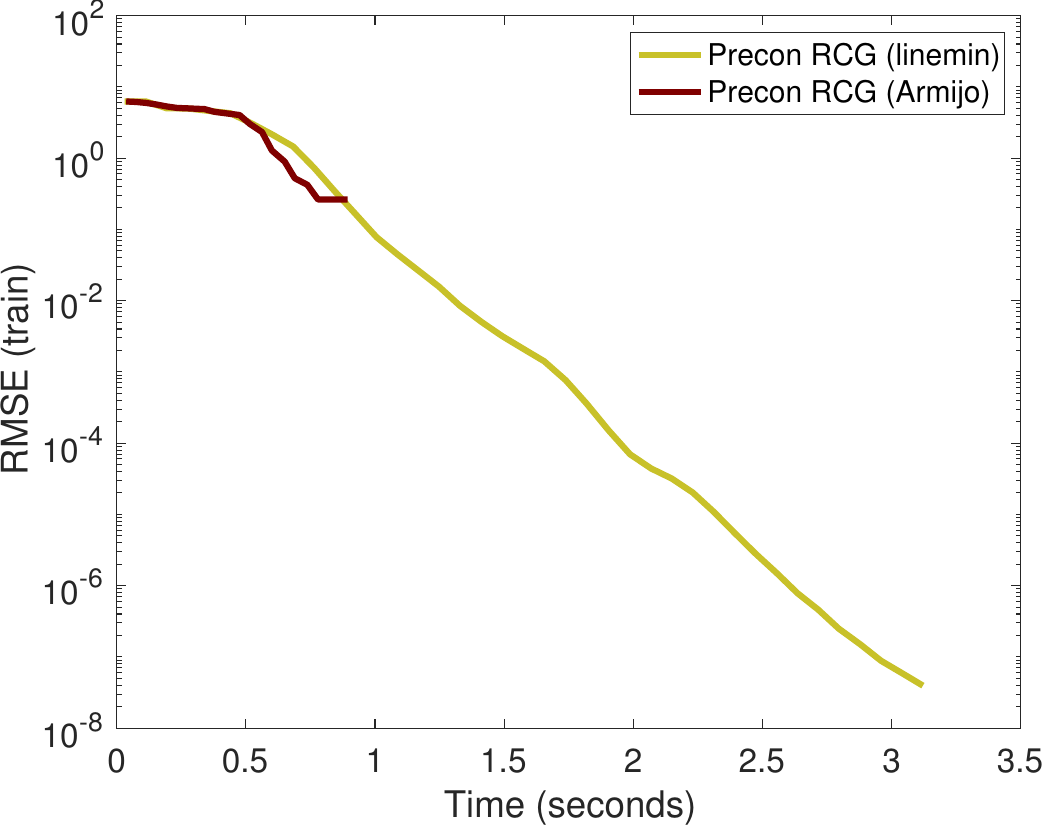}
  }
\caption{\revj{Comparisons between the stepsize selection methods. Tensor size $(100, 100, 200)$, Tucker rank $\rstar=(3, 5, 7)$.  
The rank parameter $\rkval = 14$. The sampling rate is set to $0.3$.}} 
\label{fig:gdatn0rgeqstar_comp_stepsz}
\end{figure}

\section{Experimental details}\label{appsec:exp}

\paragraph{Synthetic model}
The noise level in the synthetic tensor model~\eqref{def:datamodel0} is determined
according to the \revv{signal-to-noise ratio (SNR): $\text{SNR} = \expe[][T^{2}]/\expe[][E^2]$, %
where $T$ and $E$ are the random variables that represent the tensor entries of
the low-rank tensor $\mathcal{T}$ and the noise tensor $\mathcal{E}$
in~\eqref{def:datamodel0}. In the experiments with $T\sim\mathcal{N}(0,1)$ and
$E\sim\mathcal{N}(0,\sigma_{N})$, the parameter $\sigma_{N}$ is computed
for a given SNR. %
The SNR expressed the logarithmic decibel scale (dB) is defined as $\text{SNR (dB)}=10\log_{10}(\text{SNR})$.}

\paragraph{\revv{Initialization}} 
The initial point of all CPD and polyadic decomposition-based algorithms is a
tuple $U_{0}=(U^{(1)}_{0},\dots,U_{0}^{(k)})$, where the $m_i\times R$ factor
matrices are random Gaussian matrices: $[U_0^{(i)}]_{\ell
r}\sim\mathcal{N}(0,1)$. 
For the Tucker decomposition-based algorithm (KM16), we choose to construct an
initial point that is close enough to $U_0\in\man$ for fair comparisons. For a
Tucker rank $(r_1,\dots,r_k)$, %
we initialize KM16 $(r_1,\dots, r_k)$ with a point in Tucker decomposition form
$(G;\tilde{U}^{(1)},\dots,\tilde{U}^{(k)})$, such that its tensor representation
is close enough to $\cpdp[U_{0}]$. For this purpose, %
we set $\tilde{U}^{(i)}_{0}$ as random Gaussian matrices of size $m_i\times
r_i$ with $ \tilde{U}^{(i)}_{0}\sim\mathcal{N}(0,1)$, %
and set the core tensor $G$ of size $r_{1}\times\dots\times r_{k}$ as a random Gaussian matrix
with $G_{i_1,...,i_k}\sim\mathcal{N}(0,\sigma)$, where $\sigma=\sqrt{R/r_1...r_k}$.
The choice of this variance parameter is based on the observation that the CPD
form can be seen as a special Tucker with diagonal core tensor
$D=\text{diag}(1,...,1)\in\reals^{R\times...\times R}$. Restricting $G$ to have
the same Frobenius norm as $D$ requires that the variance parameter $\sigma =
\sqrt{R/r_1...r_k}$. In particular, in the case whre $r_i=R$, we set
$\tilde{U}_{0}^{(i)} = U_{0}^{(i)}$, for $i=1,\dots,k$. 

\paragraph{Performance evaluation}
In the experiments, we evaluate the quality of tensor completion with the root-mean-square error (RMSE), for a tensor candidate $\mathcal{T}$ and a given index set $\Omega'$, %
$\text{RMSE}(\Omega') = \fro{\mathcal{P}_{\Omega'}\left(\mathcal{T}-\tstar\right)}/\sqrt{|\Omega'|}$.
The training and test RMSE refer to $\text{RMSE}(\Omega)$ and
$\text{RMSE}(\tilde{\Omega^{c}})$ respectively, where $\Omega$ is the (training)
index set of the observed entries used in the definition of the data fitting
function $f_{\Omega}$ in~\eqref{pb:main} and the test set $\tilde{\Omega^{c}}$ is the complementary of $\Omega$ in the set of all available entries. In some of the experiments, $\tilde{\Omega^{c}}$ is a subset of the complementary set (with uniformly distributed indices) such that $|\tilde{\Omega^{c}}|=(1/4)|\Omega|$ in order to reduce the time for evaluating the test RMSE, if the whole complementary set is overwhelmingly large ($2.10^{6}$).

\section{Supplementary experiments}\label{appsec:exp-lrta}
\paragraph{Tensor recovery performances}

\revj{In addition to the tensor completion results shown in Section~\ref{ssec:exp-syn}, we
conducted tensor completion tasks on a $6\times 10$ grid of $(p, R)\in \{0.12,0.152,\dots,0.28\} \times \{4,6,\dots,16, 17, 19, 21\}$, for $5$ random runs at each $(p,R)$ (each run is based on a randomly generated index set $\Omega$).} 

\begin{figure}[htbp]
        \centering 
            \includegraphics[width=.20\textwidth]{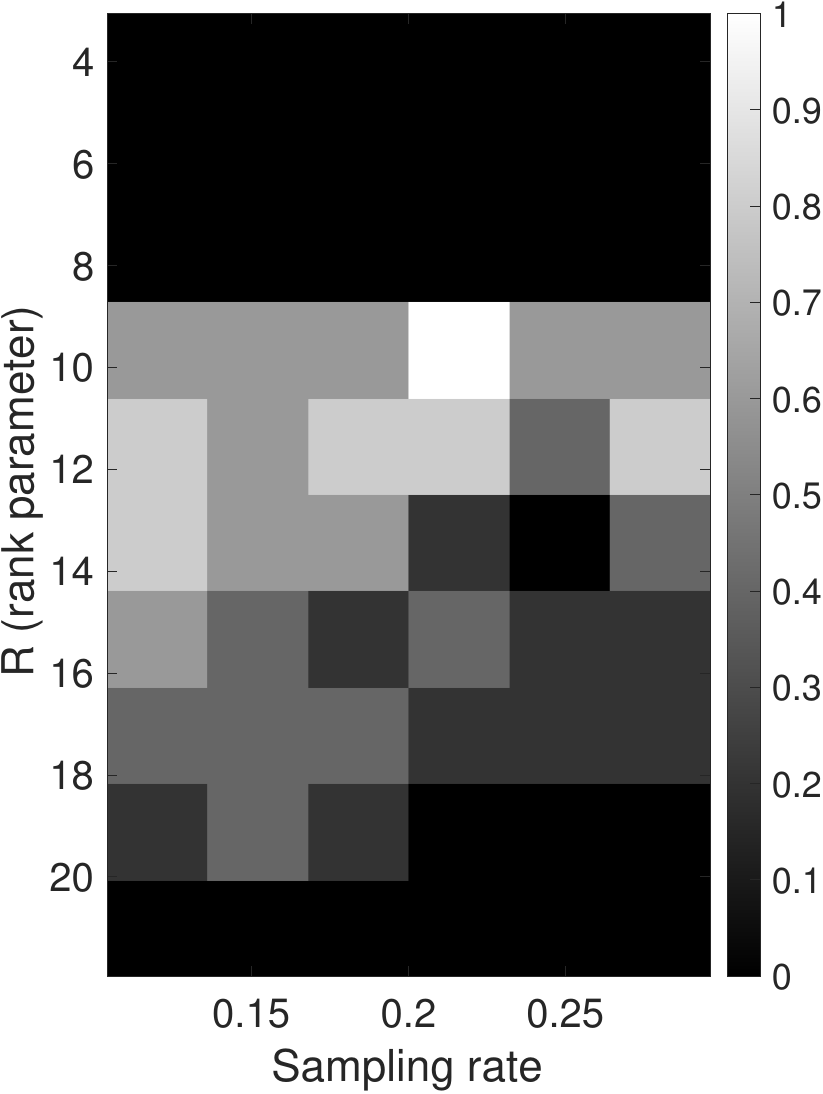} 
            \includegraphics[width=.20\textwidth]{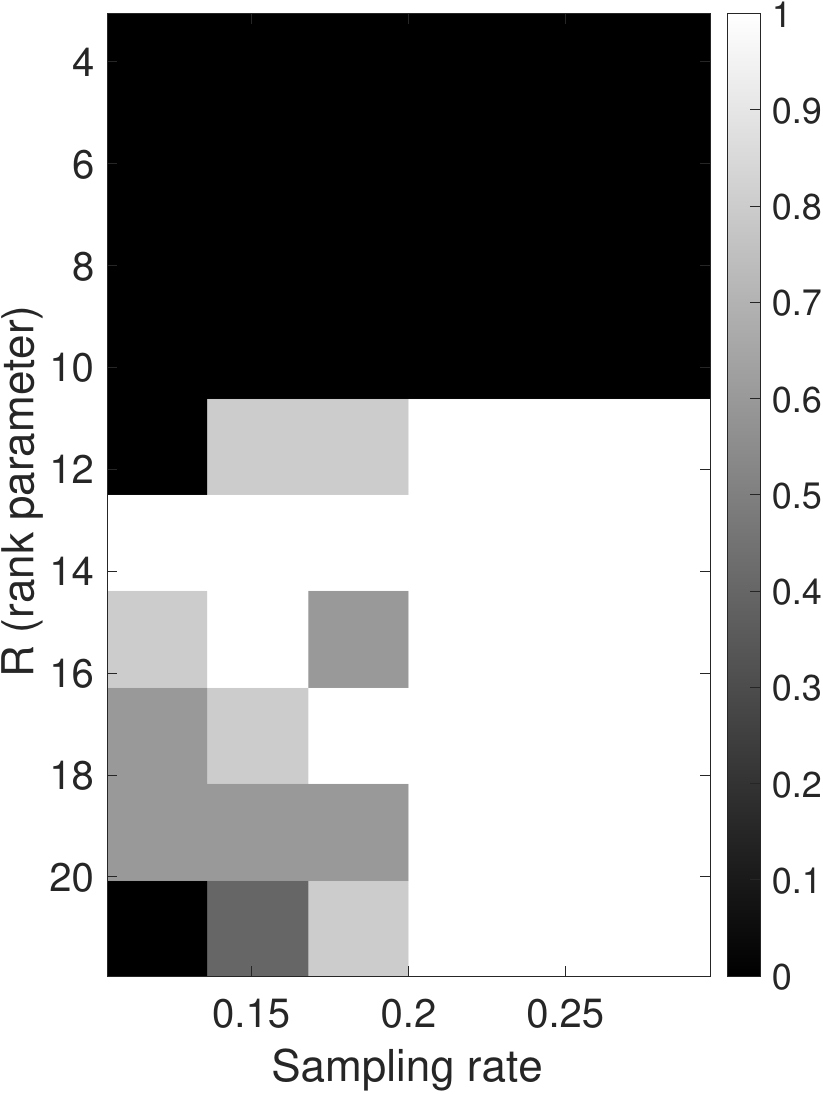} 
            \includegraphics[width=.20\textwidth]{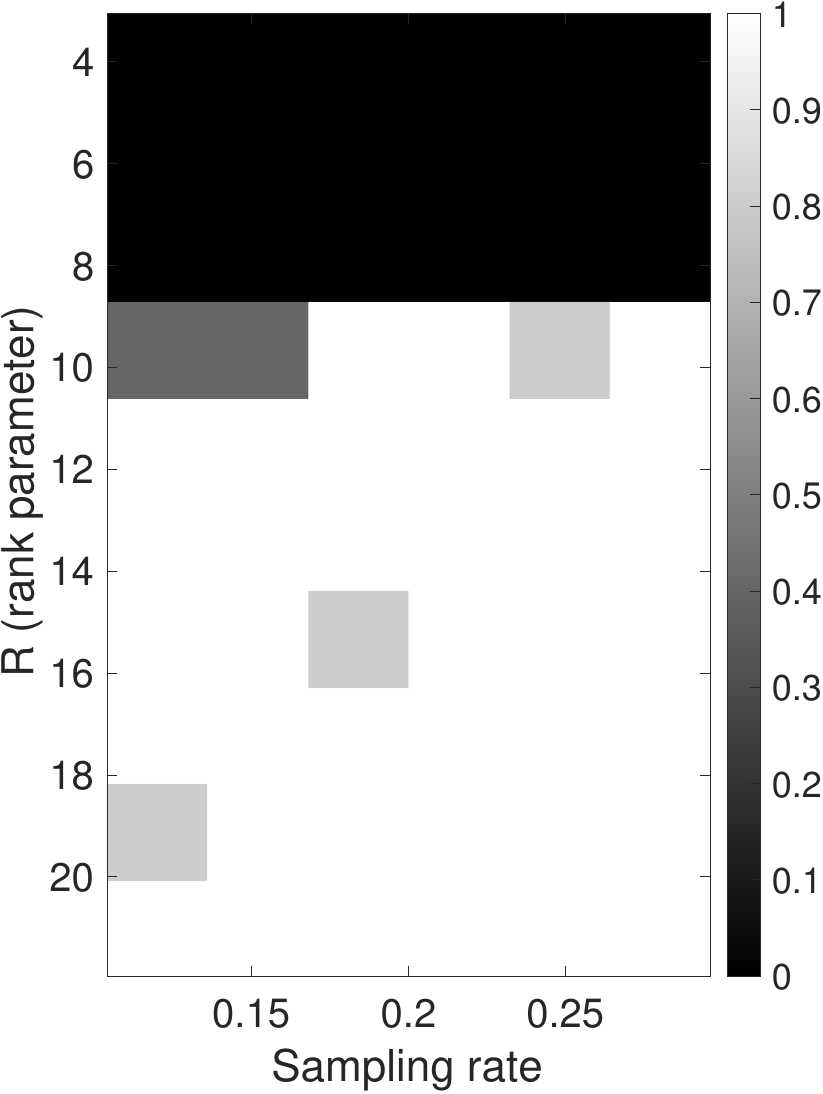} 
            \includegraphics[width=.20\textwidth]{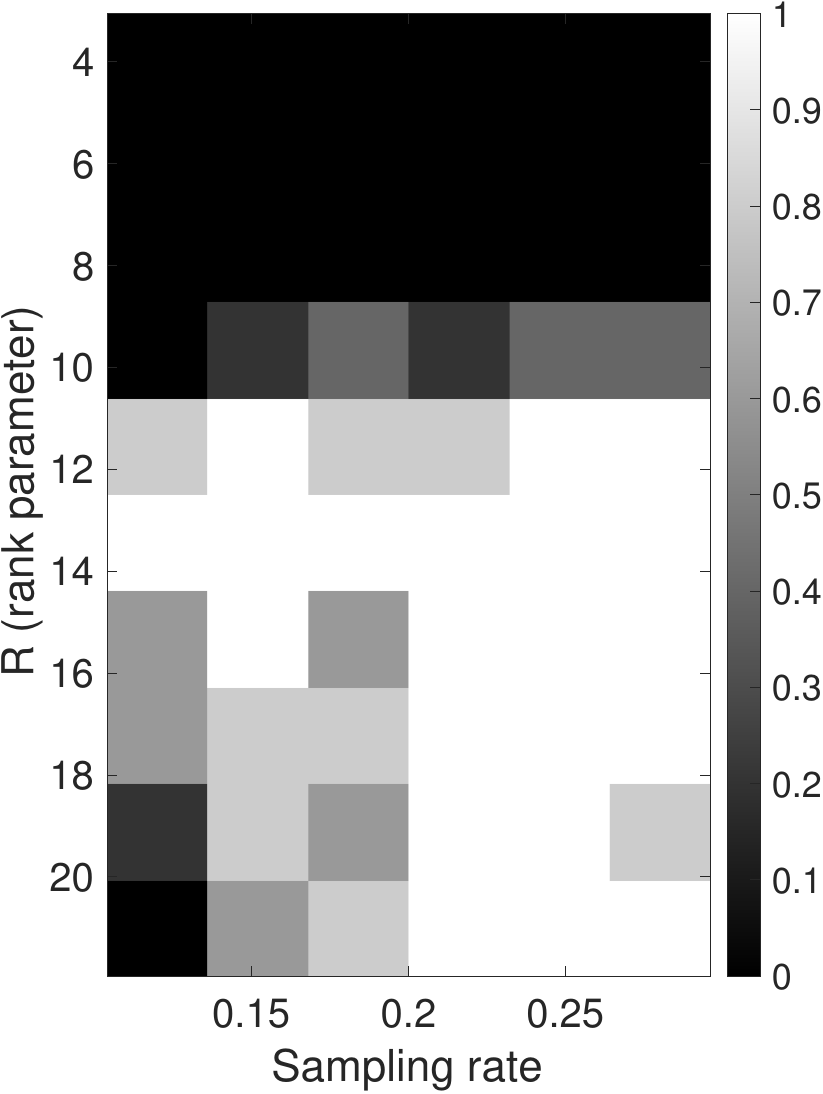}
            \\
            \subfigure[{\scriptsize\revj{Euclidean CG}}]{
            \includegraphics[width=.20\textwidth]{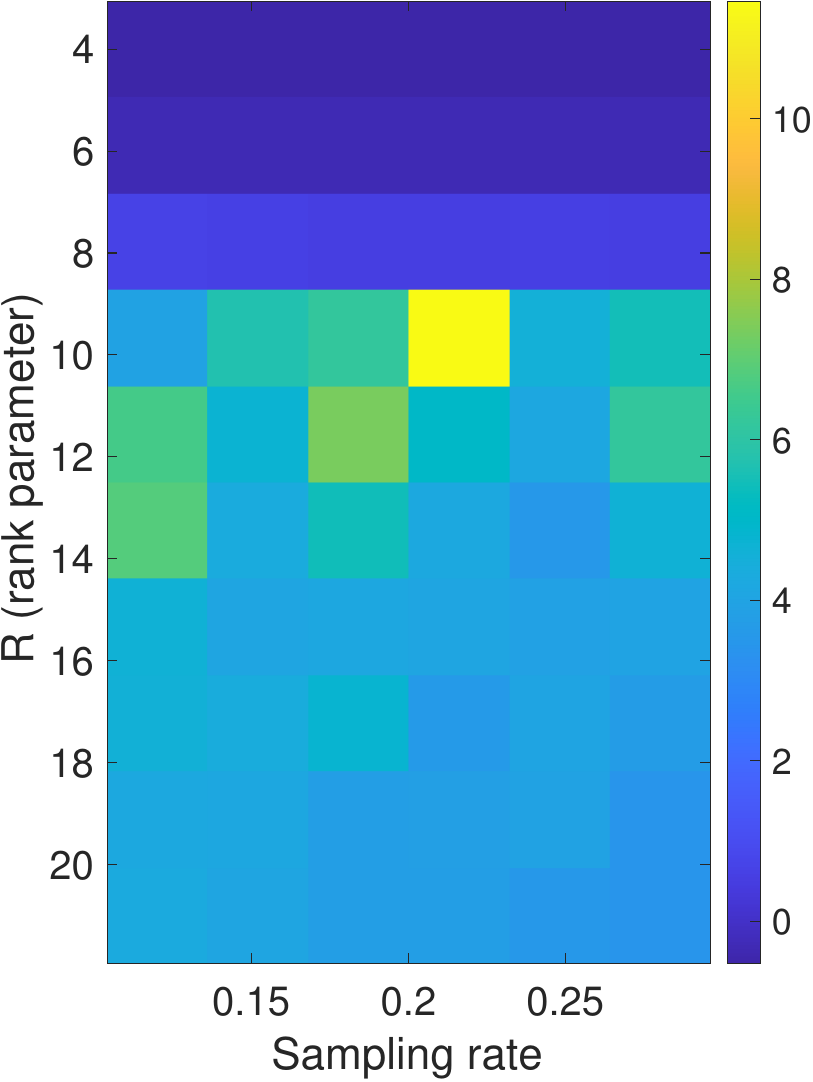} 
            }
            \subfigure[{\scriptsize\revj{PreconRGD (lmin)}}]{
            \includegraphics[width=.20\textwidth]{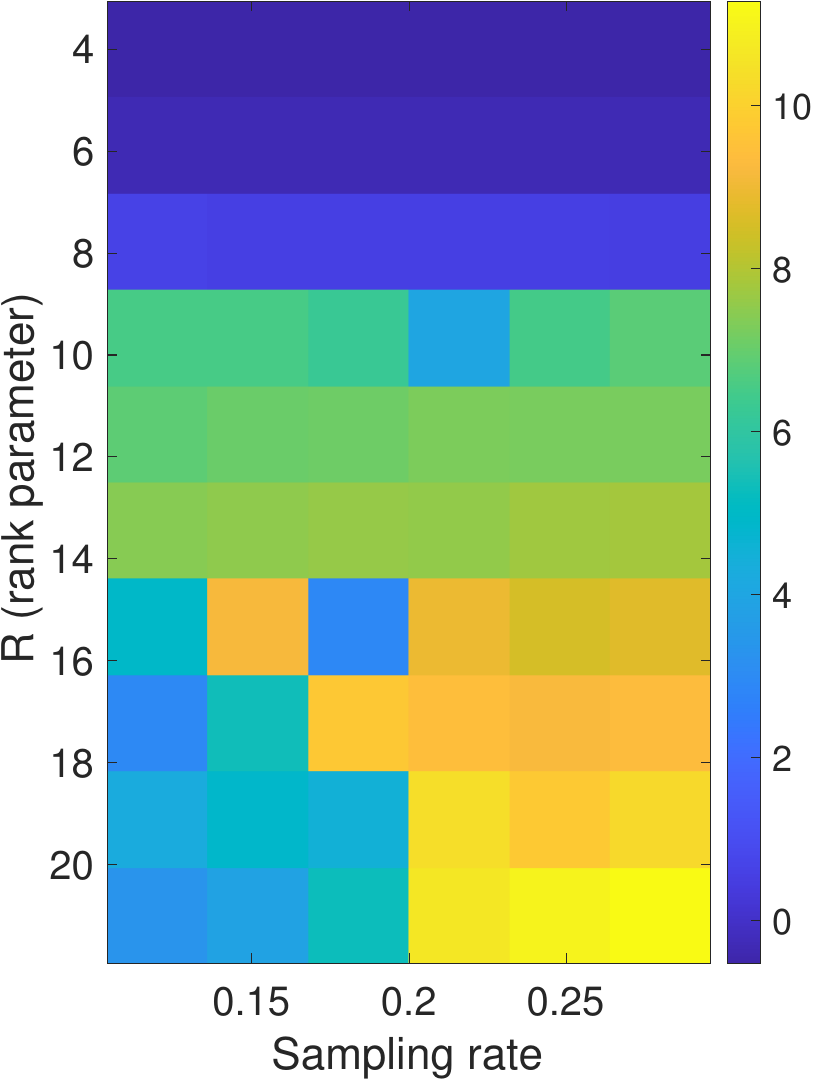} 
            }
            \subfigure[{\scriptsize\revj{PreconRCG (lmin)}}]{
            \includegraphics[width=.20\textwidth]{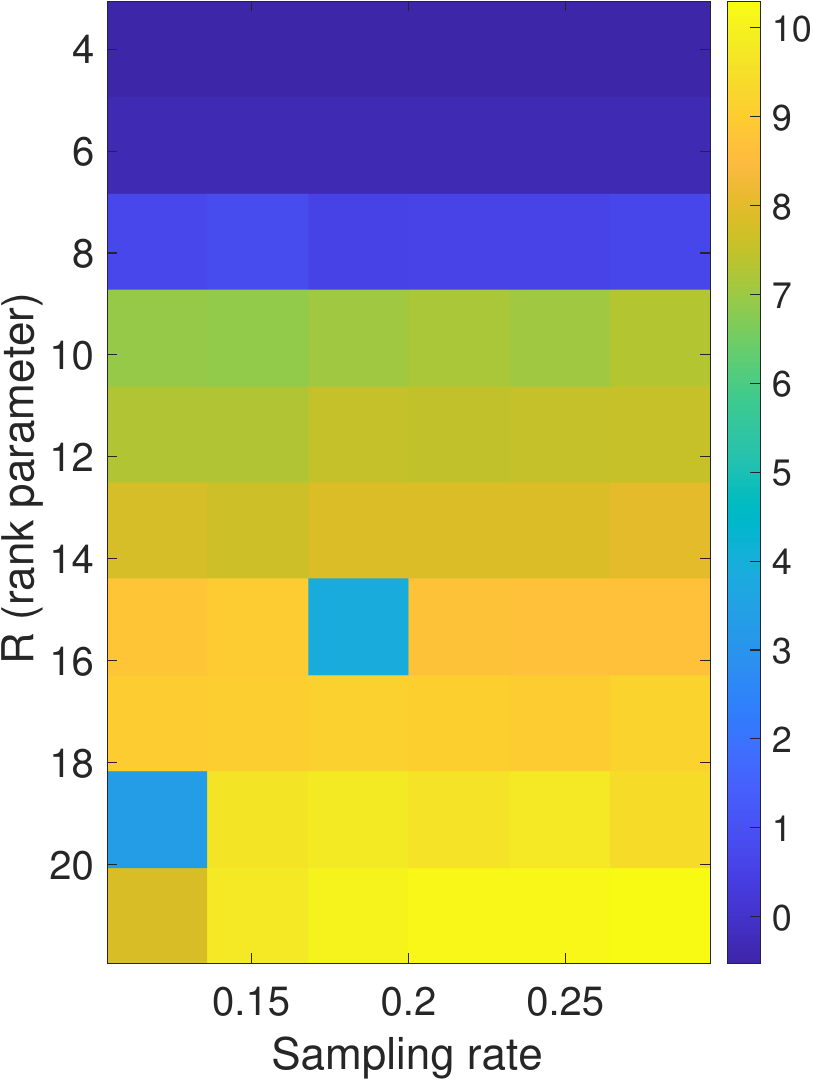} 
            }
            \subfigure[{\scriptsize\revj{PreconRGD (rbb2)}}]{
            \includegraphics[width=.20\textwidth]{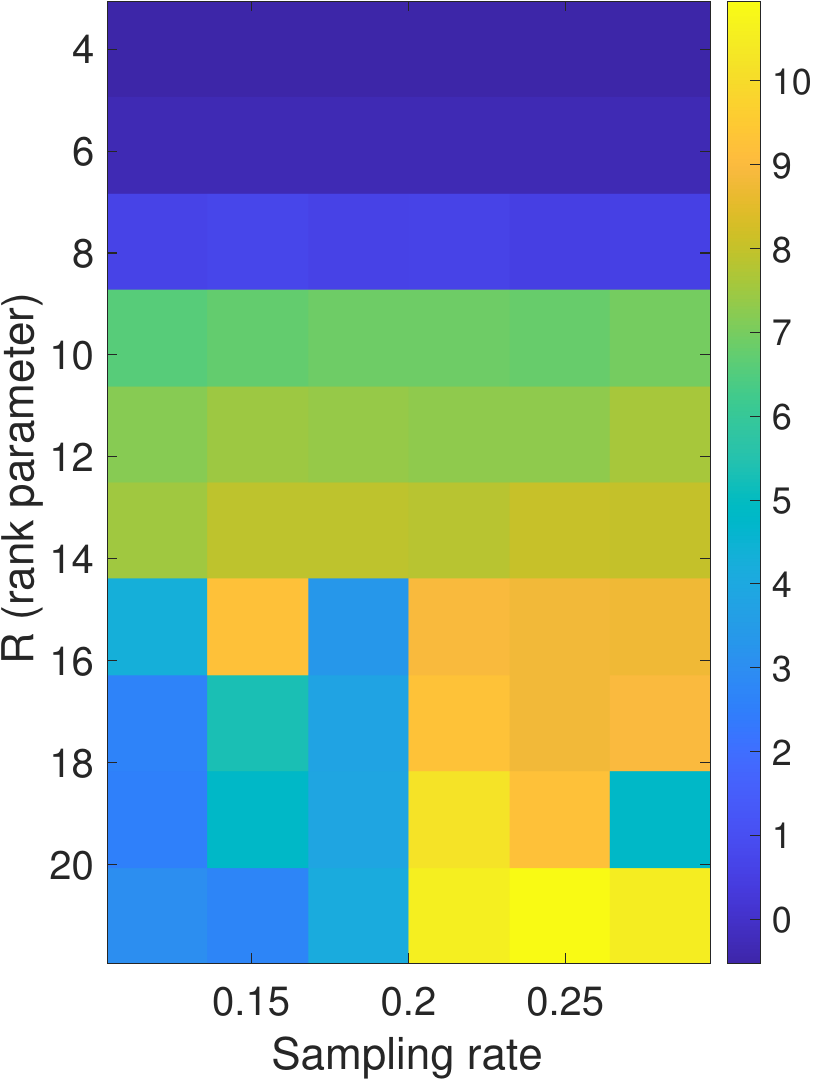} 
            }
\caption{\revj{Tensor completion performances %
based on different sampling rates and rank parameters. (Top, a-d): Rate of successful recoveries; (Bottom, a-d): Average errors in $-\log_{10}$(RMSE). The size of $\tstar$ is $(100, 100, 200)$ with a Tucker rank $\rstar=(3, 5,7)$. (lmin) is short for (linemin) in some of the algorithm descriptors above. 
    }\label{fig:synth-n0-succr}}
    \end{figure}

\revj{
\cref{fig:synth-n0-succr} shows the average final RMSEs of the solutions given
by the algorithms and the corresponding rate of
successful recoveries among the $5$ random runs. Each solution is counted as a successful recovery if its RMSE is lower than $10^{-7}$. 
In \cref{fig:synth-n0-succr}, we observe, for each algorithm tested, a clear enough phase transition of recovery rates from the unsuccessful regime to the successful regime, when the rank parameter $R$ increases. We observe that the minimal value of $R$ (for all these algorithms) to ensure a good chance of exact recovery is around $R=10$, which corresponds to a search space slightly larger but close enough in dimensionality to the tensor space with a Tucker rank $\rank_{\mathrm{tc}}(\tstar)=(3,5,7)$. }

\paragraph{Tensor approximation}
\revj{In addition to tensor completion tasks, we investigate the performance of the proposed
algorithms for low-rank CP decomposition and compare them with a low-rank CPD
algorithm using the Riemannian Gauss--Newton method~\cite{breiding2018riemannian}. } 

\begin{figure}[htbp]
        \centering 
            \includegraphics[width=.31\textwidth]{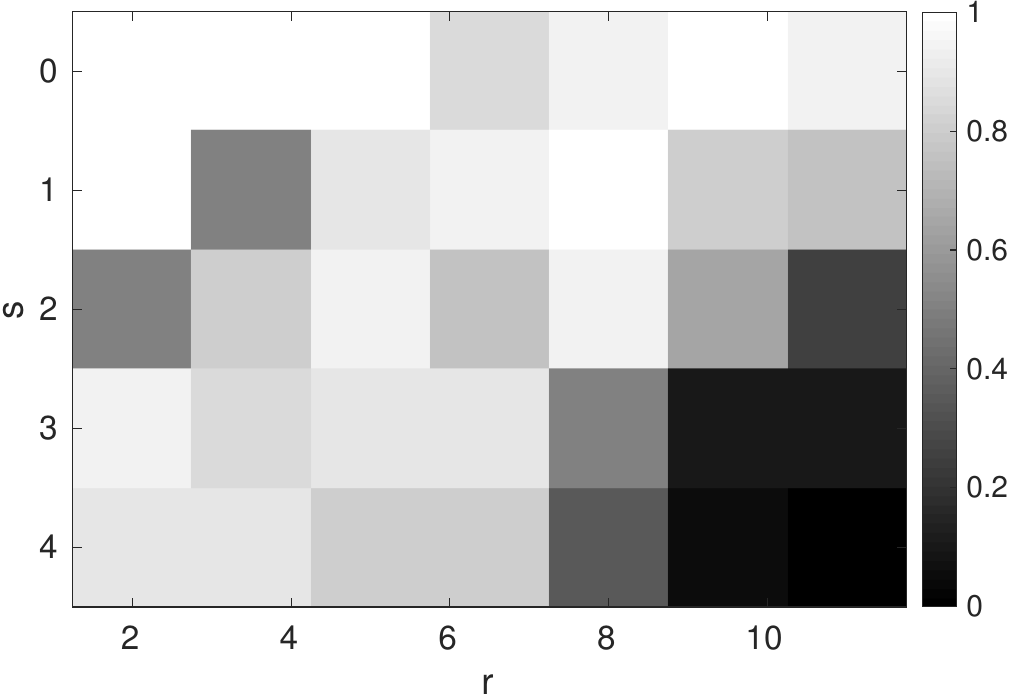}
            \quad 
            \includegraphics[width=.31\textwidth]{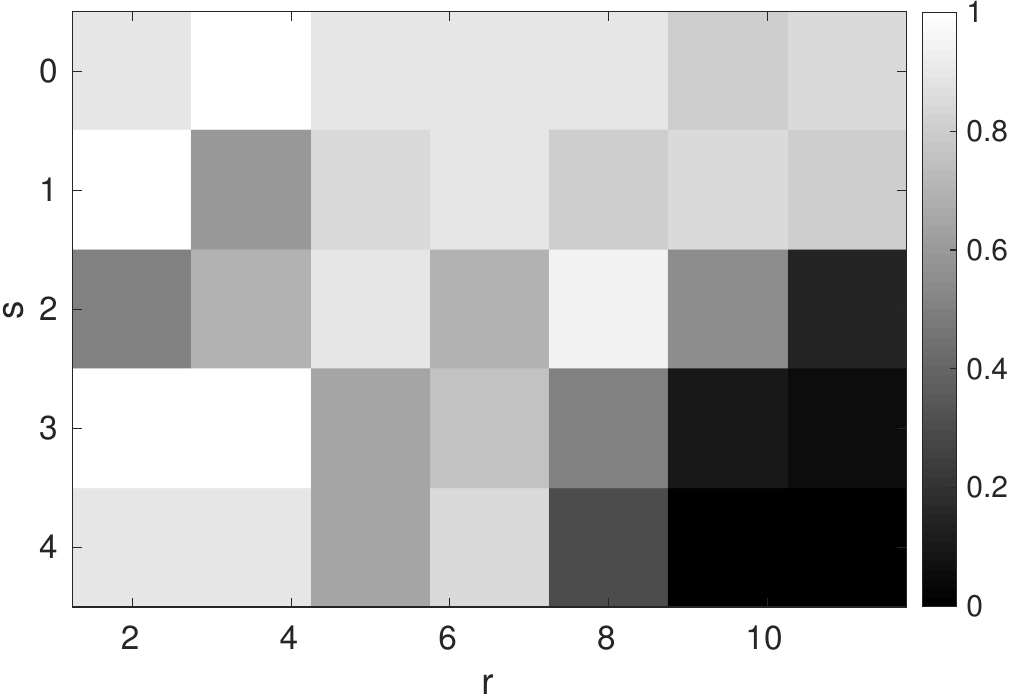}
            \quad 
            \includegraphics[width=.31\textwidth]{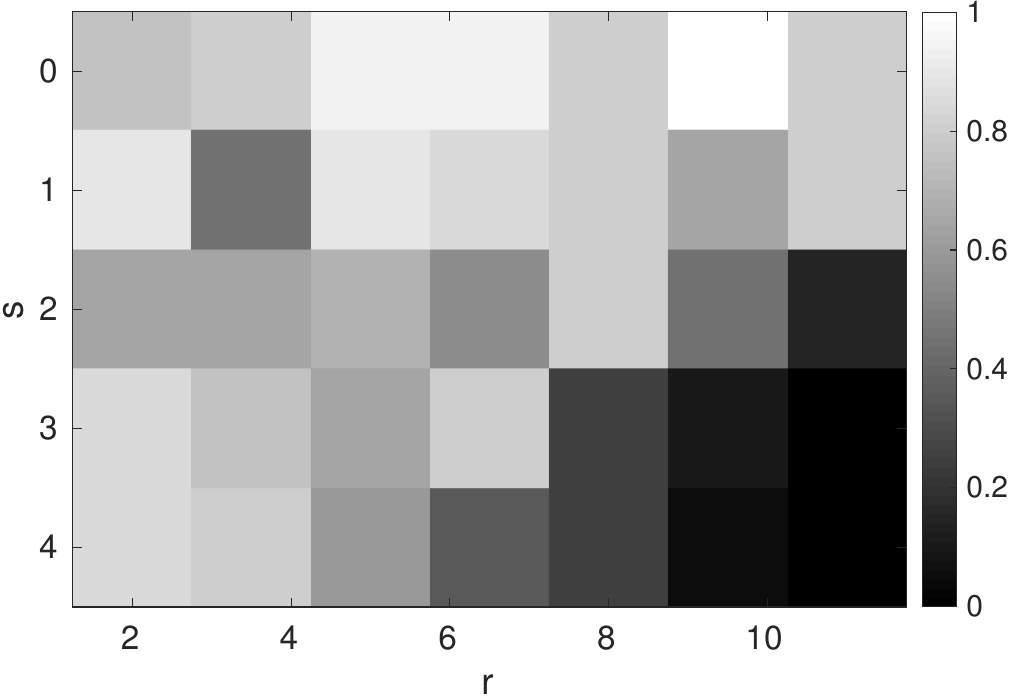} \\
            {\footnotesize (a) RGN-HR}\\~\\
            \includegraphics[width=.31\textwidth]{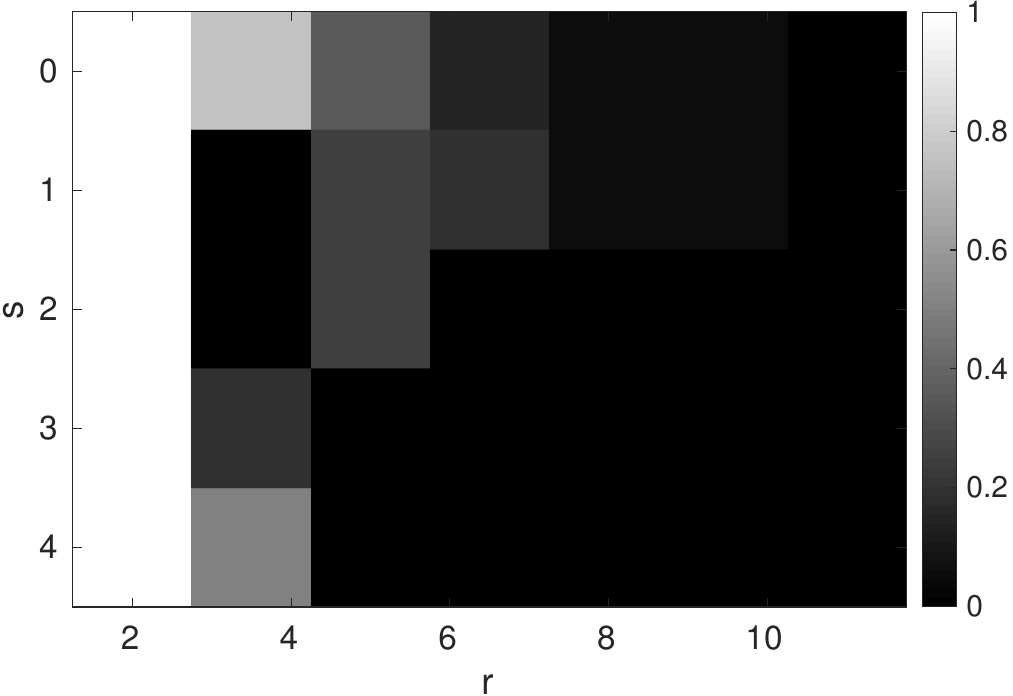}
            \quad 
            \includegraphics[width=.31\textwidth]{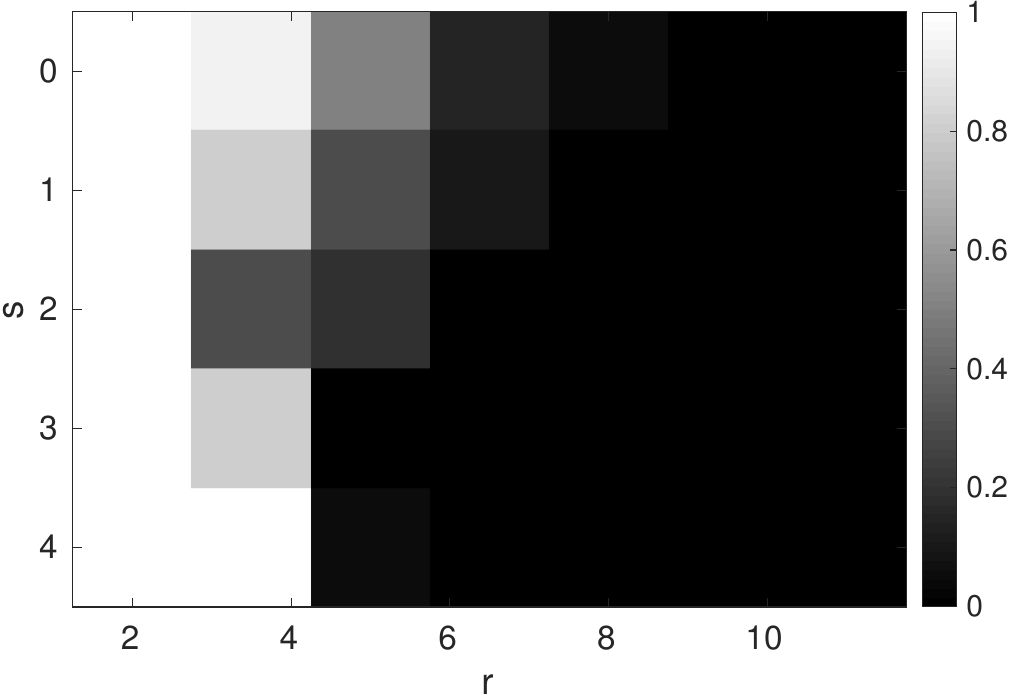} 
            \quad 
            \includegraphics[width=.31\textwidth]{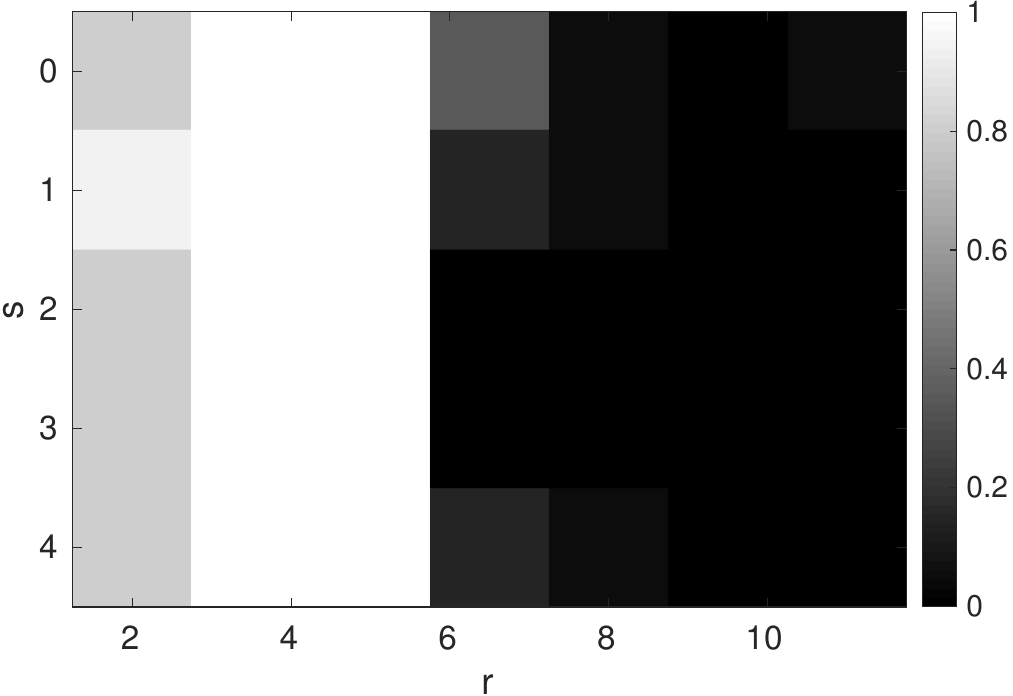} \\
            {\footnotesize (b) Precon RGD (RBB2)  }\\~\\
        \addtocounter{subfigure}{+2}
        \subfigure[{\scriptsize{($s=0$, $r=9$)}}]{
            \includegraphics[width=.4\textwidth]{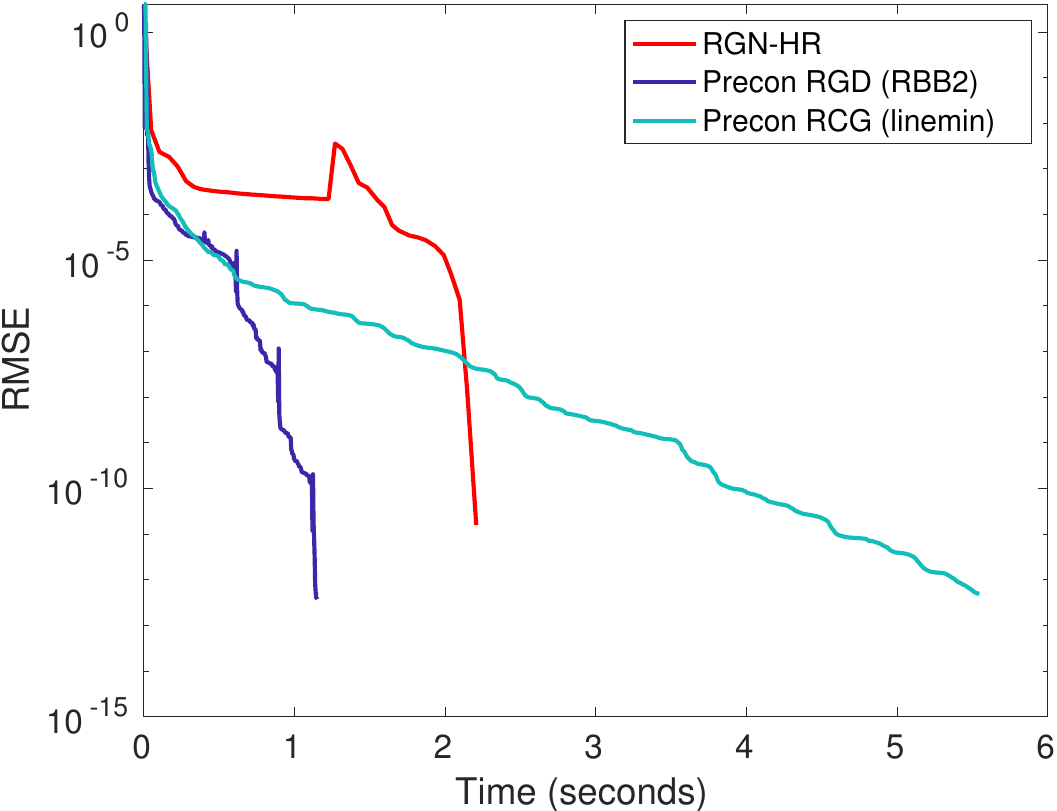}
        }
        \subfigure[{\scriptsize{($s=1$, $r=9$)}}]{
            \includegraphics[width=.4\textwidth]{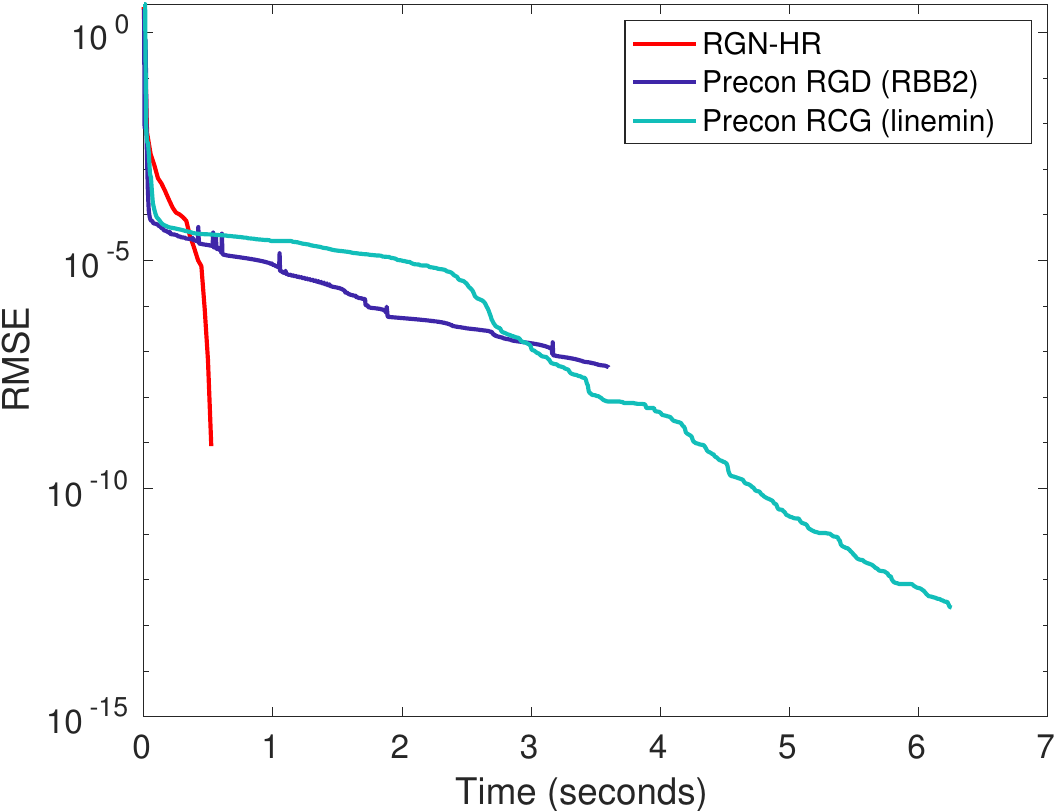}
        }
            \caption{\revj{Tensor approximation performances on tensors
            $\tstar$ generated from the class $\mc{G}_r(s)$. (a): Rate of successful attempts by RGN-HR
            using $R=\lfloor1.5r\rfloor$, $R=2r$, and $R=3r$ (from left to
            right) respectively. 
            (b): Rate of successful attempts by Precon~RGD (RBB2) using
            the same choices of $R$; %
            (c-d): Iteration
            histories %
            randomly picked from successful attempts.}\label{fig:tmp-bv18model2}} 
    \end{figure}

\revj{
The tensor approximation (TAP) task here refers to the problem of finding the best
approximation to a given, fully-observed tensor by a tensor of bounded CP rank,
which is a special case of~\eqref{prog:main} with the sampling operator reduced
to identity (where $\Omega$ is the full index set). 
The performance of these algorithms in the tensor approximation tasks is evaluated as
follows: 
(i) Randomly sample an $k^{\mathrm{th}}$-order low-rank decomposition from 
the model~\eqref{def:datamodel0} with a low Tucker-rank and the {\it Model 2} in~\cite[\S7.4]{breiding2018riemannian}, which consists of generating $A=(A^{(1)},\dots,A^{(k)}) \in \mc{G}_r(s)$ %
for a low CP rank $r$ and $s\in\{0,1,2,3,4\}$, and letting %
$\mc{A} := \cpdp[A]$; 
(ii) create a perturbed tensor $\tstar := \frac{A}{\| A\|}  + 10^{-e} \frac{E}{\|
E\|}$, where all entries of $E \sim \mc{N}(0,1)$  and $e > 0$;
(iii) solve the TAP of $\tstar$ with a rank parameter $R\geq r$, from a random
starting point $q$ using each algorithm, where $q = (Q^{(1)}, \dots, Q^{(k)}) \in
\mc{M} = R^{m_1\times R}\times \dots \times R^{m_k\times R}$ with all entries of $Q^{(i)} \sim  \mc{N}(0,1)$. 
}

\revj{We compare the performances of Precon RGD (RBB2), Precon RCG (linemin) with 
RGN-HR~\cite{breiding2018riemannian} under the two models described above. For
the proposed algorithms, we set $\delta=10^{-15}$. 
In the tests with the Model 2 of \cite[\S7.4]{breiding2018riemannian}, we explore $(r,s)$ on a $7\times 5$ grid of $\{ 2
,3     ,4     ,5     ,7     ,9    ,11\}\times\{0,1,2,3,4\}$ and test with the CP rank
parameter $R\in\{r,
\lfloor1.5r\rfloor, 2r, 3r\}$. 
The chances of successful recovery are estimated after $20$ random runs for each $(s,r)$ setting
and each value of $R$. 
The comparative results on these two models are given in Figures~\ref{fig:tmp-bv18model2}--\ref{fig:tmp-comp-bv18} respectively. }

    \begin{figure}[htbp]
        \centering 
            \includegraphics[width=.41\textwidth]{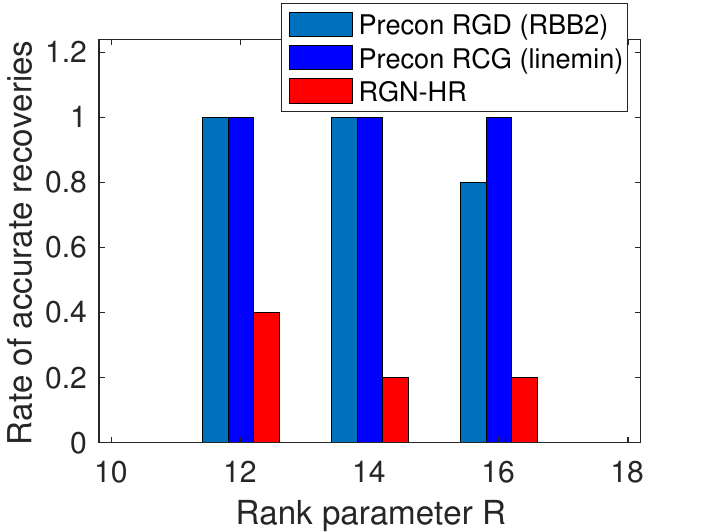} 
            \includegraphics[width=.41\textwidth]{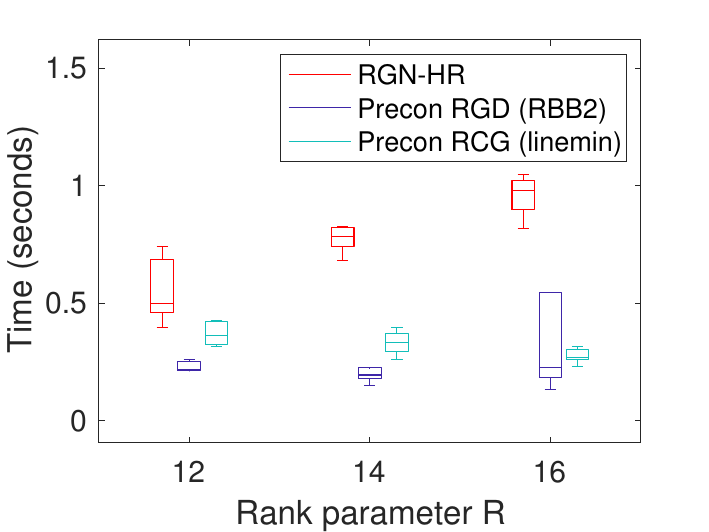}
            \\
            \includegraphics[width=.41\textwidth]{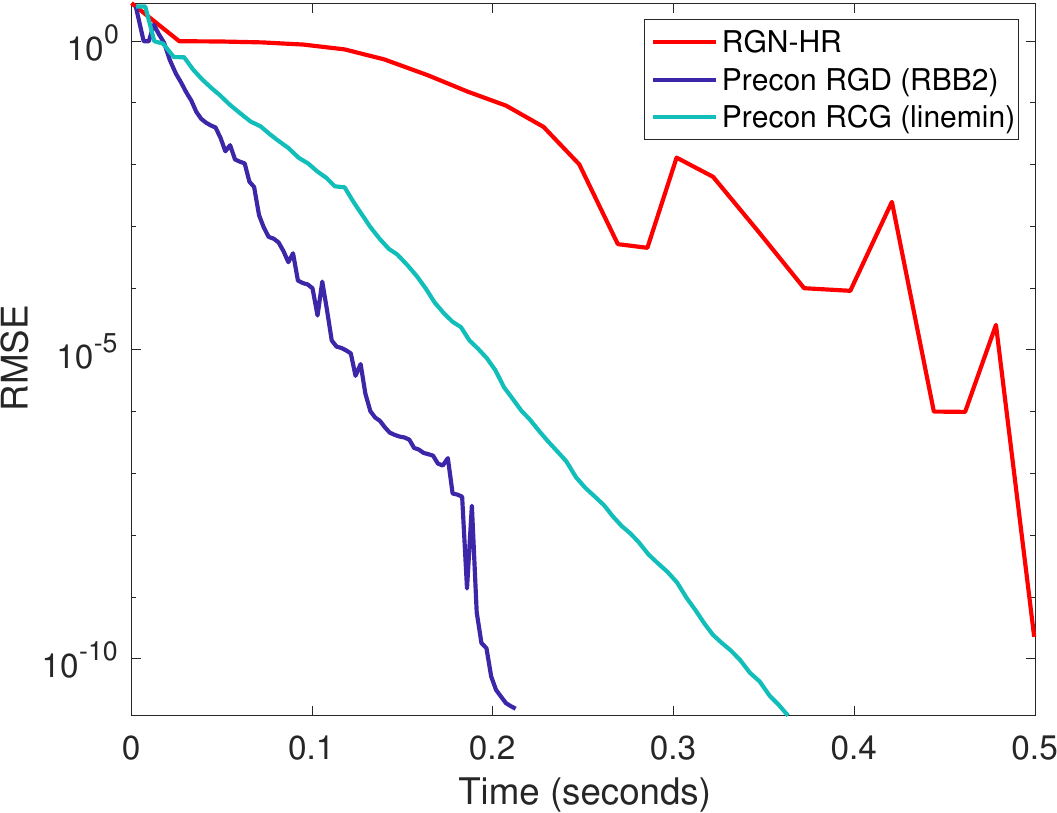} 
            \includegraphics[width=.41\textwidth]{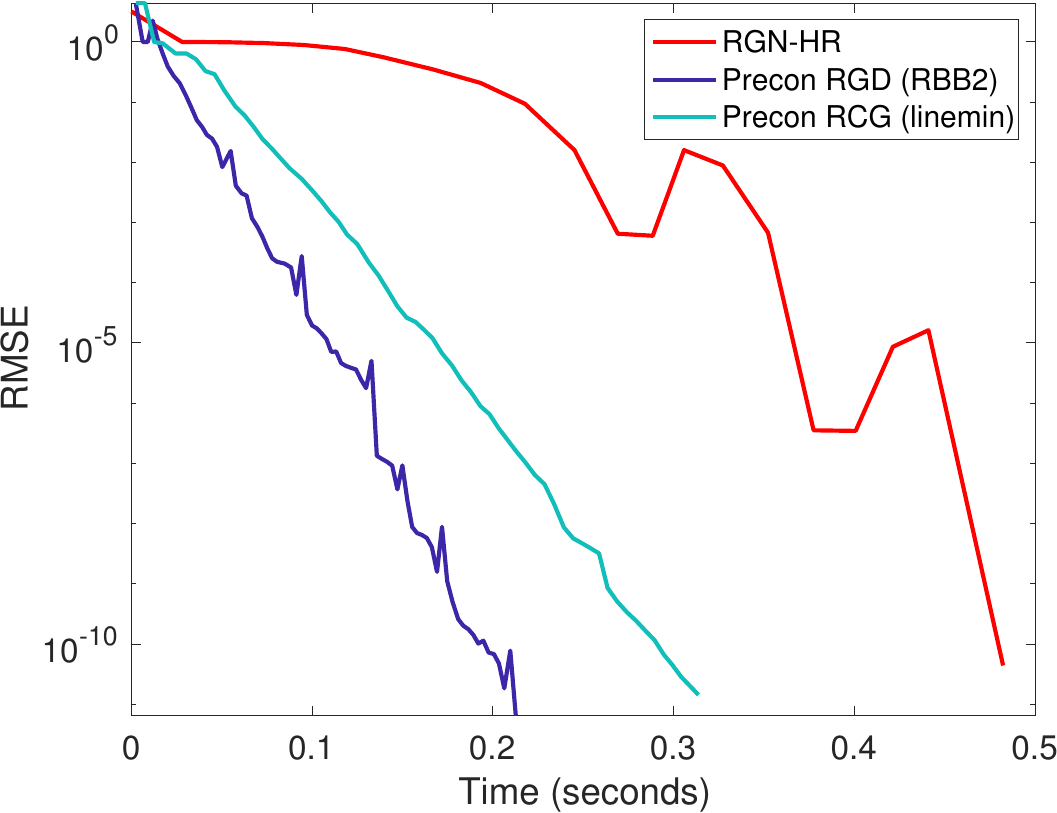} 
            \caption{\revj{{The tensor $\tstar$ is of size $10\times 30\times 30$ and has a Tucker rank $\rstar=(3,5,7)$, which is generated in the same manner as in the experiment of Fig.~2.}}\label{fig:tmp-comp-bv18}} 
    \end{figure}

\revj{From \cref{fig:tmp-bv18model2}, we observe that RGN-HS performs better than the proposed algorithms in a large part of the $7\times 5$ grid of $(r,s)$, where the tensors $\tstar$ under the Model 2 of~\cite[\S7.4]{breiding2018riemannian} are challenging because of their high condition number. 
From \cref{fig:tmp-comp-bv18}, we observe that the proposed algorithms outperform RGN-HR in both successful recovery rate and convergence time, for tensors of the model~\eqref{def:datamodel0}, where the core tensor of $\tstar$ is not cubic. 
}

\end{appendices}

\section*{Acknowledgment}
\revj{We thank the referees for the insightful comments.} 
We gratefully acknowledge the helpful discussions with {P.-A.} Absil and his valuable comments during the preparation of this paper. %


\begin{thebibliography}{10}

\bibitem{AbsMahSep2008}
{\sc P.-A. Absil, R.~Mahony, and R.~Sepulchre}, {\em Optimization Algorithms on
  Matrix Manifolds}, Princeton University Press, Princeton, NJ, 2008.

\bibitem{acar2011scalable}
{\sc E.~Acar, D.~M. Dunlavy, T.~G. Kolda, and M.~M{\o}rup}, {\em Scalable
  tensor factorizations for incomplete data}, %
  Chemometr. Intell. Lab. Syst., %
  106 (2011), pp.~41--56.

\bibitem{anandkumar2014tensor}
{\sc A.~Anandkumar, R.~Ge, D.~Hsu, S.~M. Kakade, and M.~Telgarsky}, {\em Tensor
  decompositions for learning latent variable models}, %
  J. Mach. Learn. Res., %
  15 (2014), pp.~2773--2832.

\bibitem{andersson1998improving}
{\sc C.~A. Andersson and R.~Bro}, {\em Improving the speed of multi-way
  algorithms:: Part i. tucker3}, %
  Chemometr. Intell. Lab. Syst., %
  42 (1998), pp.~93--103.

\bibitem{TTB_Software}
{\sc B.~W. Bader, T.~G. Kolda, et~al.}, {\em Matlab tensor toolbox version
  3.1}.
\newblock Available online, June 2019.

\bibitem{banco2016sampling}
{\sc D.~Banco, S.~Aeron, and W.~S. Hoge}, {\em Sampling and recovery of mri
  data using low rank tensor models}, in 2016 38th Annual International
  Conference of the IEEE Engineering in Medicine and Biology Society (EMBC),
  IEEE, 2016, pp.~448--452.

\bibitem{bertalmio2000image}
{\sc M.~Bertalmio, G.~Sapiro, V.~Caselles, and C.~Ballester}, {\em Image
  inpainting}, in Proceedings of the 27th annual conference on Computer
  graphics and interactive techniques, ACM Press/Addison-Wesley Publishing Co.,
  2000, pp.~417--424.

\bibitem{Boumal2018}
{\sc N.~Boumal, P.~A. Absil, and C.~Cartis}, {\em {Global rates of convergence
  for nonconvex optimization on manifolds}}, %
  IMA J. Numer. Anal., %
  39 (2019), pp.~1--33.

\bibitem{manopt}
{\sc N.~Boumal, B.~Mishra, P.-A. Absil, and R.~Sepulchre}, {\em {M}anopt, a
  {M}atlab toolbox for optimization on manifolds}, %
  J. Mach. Learn. Res., %
  15 (2014), pp.~1455--1459.

\bibitem{breiding2018riemannian}
\revj{{\sc P.~Breiding and N.~Vannieuwenhoven}, {\em A {Riemannian} trust region
  method for the canonical tensor rank approximation problem}, %
  SIAM J. Optim., 28 (2018), pp.~2435--2465.
  }

\bibitem{chu2009probabilistic}
{\sc W.~Chu and Z.~Ghahramani}, {\em Probabilistic models for incomplete
  multi-dimensional arrays}, in Artificial Intelligence and Statistics, 2009,
  pp.~89--96.

\bibitem{cichocki2015tensor}
{\sc A.~Cichocki, D.~Mandic, L.~De~Lathauwer, G.~Zhou, Q.~Zhao, C.~Caiafa, and
  H.~A. Phan}, {\em Tensor decompositions for signal processing applications:
  From two-way to multiway component analysis}, %
  IEEE Signal Process. Mag., 32 (2015), pp.~145--163.

\bibitem{da2013hierarchical}
{\sc C.~Da~Silva and F.~Herrmann}, {\em Hierarchical tucker tensor
  optimization-applications to tensor completion. sampta 2013}, in 10th
  International Conference on Sampling Theory and Application, Jacobs
  University Bremen, 2013.

\bibitem{DeLathauwer2000msvd}
{\sc L.~De~Lathauwer, B.~De~Moor, and J.~Vandewalle}, {\em A multilinear
  singular value decomposition}, SIAM J. Matrix Anal. Appl., 21 (2000),
  pp.~1253--1278 (electronic).

\bibitem{DeLathauwer2000:obr}
\leavevmode\vrule height 2pt depth -1.6pt width 23pt, {\em On the best rank-1
  and rank-{$(R_1,R_2,\cdots,R_N)$} approximation of higher-order tensors},
  SIAM J. Matrix Anal. Appl., 21 (2000), pp.~1324--1342 (electronic).

\bibitem{grasedyck2015alternating}
{\sc L.~Grasedyck, M.~Kluge, and S.~Kr{\"a}mer}, {\em Alternating least squares
  tensor completion in the tt-format}, arXiv preprint arXiv:1509.00311,
  (2015).

\bibitem{guan2020alternating}
{\sc Y.~Guan, S.~Dong, P.-A. Absil, and F.~Glineur}, {\em {Alternating
  minimization algorithms for graph regularized tensor completion}}, arXiv
  preprint arXiv:2008.12876,  (2020), pp.~1--30.

\bibitem{hackbusch2012tensor}
{\sc W.~Hackbusch}, {\em Tensor spaces and numerical tensor calculus}, vol.~42,
  Springer, 2012.

\bibitem{hernandez2010simple}
{\sc D.~Hernandez}, {\em Simple tensor products}, %
Invent. Math., 181 (2010), pp.~649--675.

\bibitem{Hestenes&Stiefel:1952}
{\sc M.~R. Hestenes and E.~Stiefel}, {\em {Methods of conjugate gradients for
  solving linear systems}}, %
  J. res. Natl. Bur. Stand., 49 (1952), pp.~409--436.

\bibitem{hitchcock1927expression}
{\sc F.~L. Hitchcock}, {\em The expression of a tensor or a polyadic as a sum
  of products}, Journal of Mathematics and Physics, 6 (1927), pp.~164--189.

\bibitem{iannazzo2018riemannian}
{\sc B.~Iannazzo and M.~Porcelli}, {\em The {Riemannian} {Barzilai}--{Borwein}
  method with nonmonotone line search and the matrix geometric mean
  computation}, IMA J. Numer. Anal., 38 (2018), pp.~495--517.

\bibitem{jain2014provable}
{\sc P.~Jain and S.~Oh}, {\em Provable tensor factorization with missing data},
  Adv. Neural. Inf. Process. Syst., 2014, pp.~1431--1439.

\bibitem{karatzoglou2010multiverse}
{\sc A.~Karatzoglou, X.~Amatriain, L.~Baltrunas, and N.~Oliver}, {\em
  Multiverse recommendation: n-dimensional tensor factorization for
  context-aware collaborative filtering}, in Proceedings of the fourth ACM
  conference on Recommender systems, ACM, 2010, pp.~79--86.

\bibitem{kasai2016low}
{\sc H.~Kasai and B.~Mishra}, {\em Low-rank tensor completion: a {Riemannian}
  manifold preconditioning approach}, in International Conference on Machine
  Learning, 2016, pp.~1012--1021.

\bibitem{kolda2009tensor}
{\sc T.~G. Kolda and B.~W. Bader}, {\em Tensor decompositions and
  applications}, %
  SIAM Rev., 51 (2009), pp.~455--500.

\bibitem{korah2007spatiotemporal}
{\sc T.~Korah and C.~Rasmussen}, {\em Spatiotemporal inpainting for recovering
  texture maps of occluded building facades}, %
  IEEE Trans. Image. Process., 16 (2007), pp.~2262--2271.

\bibitem{kressner2014low}
{\sc D.~Kressner, M.~Steinlechner, and B.~Vandereycken}, {\em Low-rank tensor
  completion by riemannian optimization}, %
  BIT, 54 (2014), pp.~447--468.

\bibitem{kruskal1977three}
{\sc J.~B. Kruskal}, {\em Three-way arrays: rank and uniqueness of trilinear
  decompositions, with application to arithmetic complexity and statistics},
  Linear Algebra Appl., 18 (1977), pp.~95--138.

\bibitem{kruskal1989rank}
\leavevmode\vrule height 2pt depth -1.6pt width 23pt, {\em Rank, decomposition,
  and uniqueness for 3-way and n-way arrays}, Multiway data analysis,  (1989),
  pp.~7--18.

\bibitem{liu2012tensor}
{\sc J.~Liu, P.~Musialski, P.~Wonka, and J.~Ye}, {\em Tensor completion for
  estimating missing values in visual data}, %
  IEEE Trans. Pattern. Anal. Mach. Intell., 35 (2012), pp.~208--220.

\bibitem{liu2014factor}
{\sc Y.~Liu, F.~Shang, H.~Cheng, J.~Cheng, and H.~Tong}, {\em Factor matrix
  trace norm minimization for low-rank tensor completion}, in Proceedings of
  the SIAM International Conference on Data Mining, SIAM, 2014,
  pp.~866--874.

\bibitem{Mishra2012a}
{\sc B.~Mishra, K.~Adithya, and A.~R. Sepulchre}, {\em {A Riemannian geometry
  for low-rank matrix completion}}, %
  arXiv preprint arXiv:1211.1550, (2012).

\bibitem{mishra2016riemannian}
{\sc B.~Mishra and R.~Sepulchre}, {\em Riemannian preconditioning}, %
SIAM J. Optim., 26 (2016), pp.~635--660.

\bibitem{morup2006parallel}
{\sc M.~M{\o}rup, L.~K. Hansen, C.~S. Herrmann, J.~Parnas, and S.~M. Arnfred},
  {\em Parallel factor analysis as an exploratory tool for wavelet transformed
  event-related eeg}, NeuroImage, 29 (2006), pp.~938--947.

\bibitem{nocedal2006numerical}
{\sc J.~Nocedal and S.~Wright}, {\em Numerical optimization}, Springer Science
  \& Business Media, 2006.

\bibitem{oseledets2011tensor}
{\sc I.~V. Oseledets}, {\em Tensor-train decomposition}, %
SIAM J. Sci. Comput., 33 (2011), pp.~2295--2317.

\bibitem{papalexakis2016tensors}
{\sc E.~E. Papalexakis, C.~Faloutsos, and N.~D. Sidiropoulos}, {\em Tensors for
  data mining and data fusion: Models, applications, and scalable algorithms},
ACM Trans. Intell. Syst. Technol., 8 (2016), pp.~1--44.

\bibitem{phien2016efficient}
{\sc H.~N. Phien, H.~D. Tuan, J.~A. Bengua, and M.~N. Do}, {\em Efficient
  tensor completion: Low-rank tensor train}, arXiv preprint arXiv:1601.01083,
  (2016).

\bibitem{rauhut2015tensor}
{\sc H.~Rauhut, R.~Schneider, and {\v{Z}}.~Stojanac}, {\em Tensor completion in
  hierarchical tensor representations}, in Compressed Sensing and its
  Applications, Springer, 2015, pp.~419--450.

\bibitem{rauhut2017low}
\leavevmode\vrule height 2pt depth -1.6pt width 23pt, {\em Low rank tensor
  recovery via iterative hard thresholding}, %
  Linear Algebra Appl., 523 (2017), pp.~220--262.

\bibitem{schneider2015convergence}
{\sc R.~Schneider and A.~Uschmajew}, {\em Convergence results for projected
  line-search methods on varieties of low-rank matrices via {\l}ojasiewicz
  inequality}, %
  SIAM J. Optim., 25 (2015), pp.~622--646.

\bibitem{sorber2013optimization}
{\sc L.~Sorber, M.~Van~Barel, and L.~{De~Lathauwer}}, {\em Optimization-based
  algorithms for tensor decompositions: Canonical polyadic decomposition,
  decomposition in rank-(l\_r,l\_r,1) terms, and a new generalization}, %
  SIAM J. Optim., 23 (2013), pp.~695--720.

\bibitem{sorber2015structured}
\leavevmode\vrule height 2pt depth -1.6pt width 23pt, {\em Structured data
  fusion}, %
  IEEE J. Sel. Top. Signal Process., 9 (2015), pp.~586--600.

\bibitem{sorensen2019fiber}
{\sc M.~S{\o}rensen and L.~{De~Lathauwer}}, {\em Fiber sampling approach to
  canonical polyadic decomposition and application to tensor completion}, %
  SIAM J. Matrix Anal. Appl., 40 (2019), pp.~888--917.

\bibitem{Srebro2005}
{\sc N.~Srebro, J.~D. Rennie, and T.~S. Jaakkola}, {\em {Maximum-margin matrix
  factorization}}, %
  Adv. Neural. Inf. Process. Syst., 17 (2005), pp.~1329--1336.

\bibitem{tomasi2005parafac}
{\sc G.~Tomasi and R.~Bro}, {\em Parafac and missing values}, %
  Chemometr. Intell. Lab. Syst., 75 (2005), pp.~163--180.

\bibitem{Tucker1966}
{\sc L.~Tucker}, {\em Some mathematical notes on three-mode factor analysis},
  Psychometrika, 31 (1966), pp.~279--311.

\bibitem{tensorlab}
{\sc N.~Vervliet, O.~Debals, L.~Sorber, M.~V. Barel, and L.~D. Lathauwer}, {\em
  Tensorlab 3.0}.
\newblock Available online, Mar. 2016.

\bibitem{yokota2016smooth}
{\sc T.~Yokota, Q.~Zhao, and A.~Cichocki}, {\em Smooth parafac decomposition
  for tensor completion}, %
  IEEE Trans. Signal Process., 64 (2016), pp.~5423--5436.

\end{thebibliography}
\end{document}